\documentclass[a4paper,12]{article}
\usepackage{amsmath, amsfonts, amssymb, amsthm, bm, graphics, bbm, color}
\usepackage{graphicx}
\usepackage{subfigure}
\usepackage{mathtools}
\usepackage[table]{xcolor}
\usepackage{url}
\usepackage{mathabx}
\usepackage{cancel}
\usepackage{multicol}
\usepackage{float}
\usepackage{caption}
\usepackage{mathalfa}
\usepackage{hyperref}
\usepackage{afterpage}
\usepackage{multirow}
\DeclareMathAlphabet\mathbfcal{OMS}{cmsy}{b}{n}
\newtheorem{theorem}{Theorem}[section]
\newtheorem{lemma}[theorem]{Lemma}
\newtheorem{corollary}[theorem]{Corollary}

\theoremstyle{definition}

\newtheorem{remark}[theorem]{Remark}

\newcolumntype{P}[1]{>{\centering\arraybackslash}p{#1}}

\makeatletter
\newcommand\makebig[2]{%
  \@xp\newcommand\@xp*\csname#1\endcsname{\bBigg@{#2}}%
  \@xp\newcommand\@xp*\csname#1l\endcsname{\@xp\mathopen\csname#1\endcsname}%
  \@xp\newcommand\@xp*\csname#1r\endcsname{\@xp\mathclose\csname#1\endcsname}%
}
\makeatother

\makebig{biggg} {3.0}
\makebig{Biggg} {3.5}
\makebig{bigggg}{4.0}
\makebig{Bigggg}{4.5}
\makebig{biggggg}{5.0}
\makebig{Biggggg}{5.5}
\makebig{bigggggg}{6.0}
\makebig{Bigggggg}{6.5}

\renewcommand{\vec}[1]{\mbox{\boldmath$#1$}}

\newcommand{\dif}{\mathrm{d}}
\newcommand{\im}{\mathrm{i}}

\newcommand{\vcE}{{\bm {\mathcal E}}}
\newcommand{\vcH}{{\bm {\mathcal H}}}
\newcommand{\vcJ}{{\bm {\mathcal J}}_0}
\newcommand{\vE}{{\bm {E}}}

\newcommand{\vEb}{{\bm {E}_0}}
\newcommand{\vHb}{{\bm {H}_0}}
\newcommand{\vEa}{{\bm {E}}_\alpha}
\newcommand{\vHa}{{\bm {H}}_\alpha}
\newcommand{\vEd}{{\bm {E}}_\Delta}
\newcommand{\vHd}{{\bm {H}}_\Delta}
\newcommand{\vJ}{{\bm {J}}_0}
\newcommand{\vz}{{\bm {z}}}
\newcommand{\vn}{{\bm {n}}}
\newcommand{\ve}{{\bm {e}}}
\newcommand{\va}{{\bm {a}}}
\newcommand{\vb}{{\bm {b}}}
\newcommand{\vx}{{\bm {x}}}
\newcommand{\vy}{{\bm {y}}}
\newcommand{\vu}{{\bm {u}}}
\newcommand{\vw}{{\bm {w}}}
\newcommand{\vwz}{{\bm {w}}_0}
\newcommand{\vxi}{{\bm {\xi}}}
\newcommand{\vv}{{\bm {v}}}
\newcommand{\vF}{{\bm {F}}}
\newcommand{\vD}{{\bm {D}}}

\newcommand{\vth}{{\bm {\theta}}}
\newcommand{\vpsi}{{\bm {\psi}}}
\newcommand{\vphi}{{\bm {\phi}}}
\newcommand{\vzero}{{\bm {0}}}
\newcommand{\vX}{{\bm {X}}_\alpha}
\newcommand{\vR}{{\bm {R}}}
\newcommand{\vT}{{\bm {T}}}
\newcommand{\vr}{{\bm {r}}}

\DeclareMathOperator*{\esssup}{ess \, sup}

\begin{document}
\title{Characterising small objects in the regime between  {the eddy current model and wave propagation}}

\author{P.D. Ledger$^\dagger$  and W.R.B. Lionheart$^\ddagger$\\
$^\dagger$School of Computer Science \& Mathematics, Keele University \\
$^\ddagger$Department of Mathematics, The University of Manchester\\
Corresponding author: p.d.ledger@keele.ac.uk}

\maketitle
\section*{Abstract}

Being able to characterise objects at low frequencies, but above the limiting frequency of the eddy current approximation of the Maxwell system, is important for improving current metal detection technologies. Importantly, the upper frequency limit of the eddy current model depends on the object topology and on its materials, with the maximum frequency being much smaller for certain geometries compared to others of the same size and materials. Additionally, the eddy current model breaks down at much smaller frequencies for highly magnetic conducting materials compared to non-permeable objects (with similar conductivities, sizes and shapes) and, hence, characterising small magnetic objects made of permeable materials using the eddy current at typical frequencies of operation for a metal detector  is not always possible. To address this, we derive a new asymptotic expansion for permeable highly conducting objects that is valid for small objects  {and holds for frequencies just} beyond the eddy current limit. The leading order term we derive leads to new forms of object characterisations in terms of polarizability tensor object descriptions where the coefficients can be obtained from solving vectorial transmission problems.
We expect these new characterisations to be important when considering objects at greater stand-off distance from the coils, which is important for safety critical applications, such as the identification of landmines, unexploded ordnance and concealed weapons. We also expect our results to be important when characterising artefacts of archaeological and forensic significance at greater depths than the eddy current model allows and to have further applications parking sensors and improving the detection of hidden, out-of-sight, metallic objects.

\noindent{\bf Keywords:} Asymptotic analysis, time harmonic Maxwell, eddy current, inverse problems, magnetic polarizability tensor, metal detection.

\noindent {\bf 2020 Mathematics Subject Classification}: 35R30; 35B30; 35Q61; 78A25; 78A46

\section{Introduction}

Improving the characterisation of metal objects from electromagnetic field perturbations, when their location is not in the immediate proximity of the exciting and sensing coils, is important for a range of applications. For example, while traditional walk-through metal detectors are used in transport hubs, court rooms, museums, galleries and at concerts and public events for identifying potential threat objects in close proximity (typically much less than a  {metre})  of the coils, there is also considerable interest from security and police forces in being able to identify potential terrorist threat objects, including guns, knives and improvised weapons, at greater stand-off distances (in the order of several metres).
The ability to characterise objects at greater stand-off distances could also bring benefits to other metal detection applications including improving safety for anti-personal landmine disposal teams, by identifying  unexploded ordnance and landmines at greater distances,  improving the detection of smaller, and deeper buried  artefacts of archaeological significance and improving the detection  of items  in forensic searches, which are beyond the reach of current metal detection technologies. There are further applications in parking sensors and in the detection of out-of-vision objects, such as cyclists, for large vehicles, such as lorries and busses.

The metal detection problem is usually described by the eddy current approximation of the Maxwell system, since the conductivities of the objects of interest are high and the frequencies of approximation are low.  The eddy current approximation neglects the displacement currents in the Maxwell system, which are responsible for wave propagation and is often applied if the wavelength  is large compared to the object size. Ammari, Buffa and N\'ed\'elec~\cite{ammaribuffa2000} show, depending on the topology of the object and its materials, that the eddy current approximation is either  first or second order accurate when considering an asymptotic expansion of the electric and magnetic fields  as the excitation
frequency tends to zero. In practice, this means that the eddy current model will  { break} down at a lower frequency for a horse shoe shaped conductor compared to a similar sized sphere made of the same material. Since the wavelength needs to be large compared to the object size,  this also means that the eddy current model breaks down at a 
lower frequency compared to non-magnetic objects of the same shape, size and conductivity. The modelling error in the eddy current approximation has been considered by Schmidt, Hiptmair and Sterz~\cite{schmidteddycurrent} who find that the approximation is accurate provided that
\begin{align}
C_1 \epsilon_* \mu_* \omega^2 \alpha^2 \ll 1,  \qquad C_2 \frac{\omega \epsilon_*}{ \sigma_*} \ll 1, \label{eqn:eddylimithip}
\end{align}
where $\epsilon_*$, $\mu_*$ and $\sigma_*$ denote the conducting object's permittivity, permeability and conductivity, respectively, $\omega$ is the angular frequency and $\alpha$ is a measure of the object size. The constants $C_1$ and $C_2$ depend on the topology of the object and, if they are of a moderate size, (\ref{eqn:eddylimithip}) reduces approximately down to the condition that the wavelength needs to be large compared to the object size (quasi-static limit) and $\epsilon_* \omega \ll \sigma_*$ (conductivities are large), respectively. They also provide a procedure for estimating $C_1$ and $C_2$.

For the eddy current problem, Ammari, Chen, Chen, Volkov and Garnier~\cite{Ammari2014} have derived the leading order term in an asymptotic expansion for the perturbation of the magnetic field caused by presence of a permeable conducting object as it size tends to zero for the eddy current problem. There are conditions imposed on the material parameters, which allows this result to be applied to the case of constant $\sigma_*$ and $\mu_*$ as $\alpha \to 0$.
 We have  {shown,~\cite{LedgerLionheart2015},} that the leading order term they have been derived can be simplified and provides an object characterisation, in terms of a rank 2 magnetic polarizability tensor (MPT) that depends on the object shape, object size, its material parameters and the frequency of excitation, but is independent of its position. The object location is separated from the other object characterisation information with the background field being evaluated at the object position. 
In~\cite{LedgerLionheart2018g} we have extended the leading order term obtained by Ammari {\it et al} in~\cite{Ammari2014} to a complete expansion of the perturbed magnetic field and introduced the concept of generalised MPTs, which provide improved object characterisation.
We have derived alternative forms of the MPT and considered spectral behaviour in~\cite{LedgerLionheart2019}. Together with Wilson, we have proposed an efficient computational procedure for the computation of the MPT spectral signatures~\cite{ben2020} and shown how it can be used characterise threat and non-threat objects of relevance to security screening~\cite{ledgerwilsonamadlion2021} and proposed an approach to classification using machine learning~\cite{benclass}.

For the full Maxwell system, Ammari, Vogelius and Volkov~\cite{ammari2001} have derived the leading order term in an asymptotic expansion of the perturbed electric and magnetic fields caused by the presence of an object as its size tends to zero. In this case, there is no simplification in the Maxwell system and wave propagation effects are included. The resulting expansions shows that leading order provides an object characterisation is in terms of the simpler P\'oyla-Szeg\"o tensor parameterised by contrasts in (complex) permittivity and permeability, which can be computed from a scalar transmission problem.  Unlike the eddy current case, conditions are not imposed on the material parameters and, hence, $\epsilon_*, \sigma_*$ and $\mu_*$ are allowed to change as $\alpha \to 0$.  {In the limiting case of magnetostatics},  it can be shown that the MPT  simplifies to the simpler P\'oyla-Szeg\"o tensor description~\cite{LedgerLionheart2019},  {whose coefficients are obtained by solving scalar transmission problems, but in the eddy current regime the calculation of the MPT coefficients} requires the solution of vectorial curl-curl transmission problems. However, it is not clear what form the characterisation takes in the regime just beyond the eddy current limit.

Our present work contributes to understanding the regime where frequencies are small, but the displacement currents remain important through the following novelties:

\begin{enumerate}

\item We derive a new asymptotic formula for the magnetic field perturbation caused by the presence of a conducting permeable object as its size tends to zero. We establish the leading order for the full Maxwell system with specific constraints applied to the object's materials and the frequency of excitation.

\item  {We obtain new forms of polarizability tensors for characterising objects}, whose coefficients are obtained by post-processing solutions to vectorial transmission problems.

\item We show that our expansion reduces to the result obtained by Ammari {\it et al}~\cite{Ammari2014} if the eddy current approximation is made.  {Our new polarizability tensors reduce to an 
object characterisation using the MPT in this case.}

\item Additionally, we obtain an expansion of the form obtained by Ammari {\it et al}~\cite{ammari2001} if the conditions we impose are relaxed.

\end{enumerate}

The material proceeds as follows: In Section~\ref{sect:notation} we introduce some notation that will aid the presentation of the material. Then, in Section~\ref{sect:mathmodel} we introduce the mathematical model that will be of interest in this work. Section~\ref{sect:main} presents our main result, the proof of which will be established subsequent sections. Section~\ref{sect:energy} establishes some energy estimates, Section~\ref{sect:integral} establishes an integral representation formula and Section~\ref{sect:proof} the remaining parts of the proof. Then, Section~\ref{sect:altform} presents alternative forms of our main result and presents simplifications and relaxation of assumptions that make comparisons with previous object characterisations possible.

\section{Notation} \label{sect:notation}
For what follows  it is beneficial to introduce the following notation: We will use boldface for vector quantities  (e.g. $\vu$) and denote by $\ve_j$, $j=1,2,3$ the units vectors associated with an orthonormal coordinate system. We denote the $j$-th component of a vector $\vu $ in this coordinate system by $(\vu )_j = \vu  \cdot \ve_j = u_j$. 
We will use calligraphic symbols to denote rank 2 tensors,  e.g. ${\mathcal N} = {\mathcal N}_{ij} \ve_i \otimes \ve_j$, where Einstein summation convention is implied, and denote their coefficients by ${\mathcal N}_{ij}$. 

We recall that for $0\le \ell < \infty$, $0\le p< \infty$,
\begin{equation}
\| \vu  \|_{W^{\ell ,p}(B_\alpha)} : = \left ( \sum_{j=0}^\ell  \int_{B_\alpha} | \vD^j( \vu (\vx ))|^p \dif \vx \right)^{1/p}, \nonumber 
\end{equation}
where the derivatives are defined in a weak sense and 
\begin{equation}
\| \vu  \|_{W^{\ell ,\infty}(B_\alpha)} : =\esssup_{\vx  \in B_\alpha} \sum_{j=0}^\ell | \vD^j( \vu (\vx ))| \nonumber .
\end{equation}

\section{Mathematical Model} \label{sect:mathmodel}
\subsection{Governing equations}

For linear materials, the time harmonic Maxwell equations are
\begin{subequations}
\begin{align}
\nabla \times \vcE & = \im \omega \mu_0 \vcH, \\
\nabla \times \vcH  &=  \sigma \vcE + \vcJ -\im \omega \epsilon \vcE, \\ 
\nabla \cdot ( \epsilon \vcE )  & = \frac{1}{\im \omega} \nabla \cdot ( \sigma \vcE ) , \\
\nabla \cdot ( \mu \vcH ) & = 0 ,
\end{align}
\end{subequations}
where $\vcE$ and $\vcH$ denotes the complex amplitudes of the electric and magnetic field intensity vectors, respectively, for an assumed $e^{-\im \omega t}$ time variation with angular frequency $\omega>0 $ and $\im =\sqrt{-1}$. In addition, $\vcJ$ denotes the complex amplitude of an external solenoidal  current source and the parameters $\epsilon, \mu$ and $\sigma$ denote the permittivity, permeability and conductivity, respectively, and satisfy
\begin{align}
0 < \epsilon^{\text{min}} \le \epsilon \le \epsilon^{\text{max}} < \infty, \qquad 0 < \mu^{\text{min}} \le \mu \le \mu^{\text{max}} < \infty, \qquad  0 \le \sigma \le \sigma^{\text{max}} < \infty .
\end{align}
As standard (e.g. Monk~\cite{monk}), the scaled fields are introduced as $\vEa = \epsilon_0^{1/2} \vcE$, $\vHa = \mu_0^{1/2} \vcH$, $\vJ = \mu_0^{1/2}\vcJ$, where $\epsilon_0 = 8.854 \times 10^{-12} $ F/m and $\mu_0 = 4\pi \times 10^{-7}$ H/m are the free space values of the permittivity and permeability, respectively, leading to 
\begin{subequations}\label{eqn:maxwell}
\begin{align}
\nabla \times \vEa & = \im k \tilde{\mu}_r \vHa,   \\
\nabla \times \vHa & = \vJ - \im k \tilde{\epsilon}_r \vEa,  \\
\nabla \cdot ( \tilde{\epsilon}_r \vEa )  & = 0 , \\
\nabla \cdot ( \tilde{\mu}_r \vHa )  & = 0,
\end{align} 
\end{subequations}
where
\begin{align}
\tilde{\epsilon}_r: = \frac{1}{\epsilon_0} \left ( \epsilon + \frac{\im \sigma}{\omega}  \right ) , \qquad \tilde{\mu}_r : = \frac{\mu}{\mu_0} , \qquad k:= \omega ( \epsilon_0 \mu_0) ^{1/2} , \nonumber 
\end{align}
and $\tilde{\epsilon}_r, \tilde{\mu}_r$ are, in general, functions of position.

\subsection{Perturbed field formulation}

We describe a single connected inclusion by $B_\alpha: = \alpha B + {\bm z}$, which means that it could be thought of a unit-sized object $B$ located at the origin, scaled by $\alpha$ and translated by ${\bm z}$.
Its boundary, $\partial B_\alpha$, is equipped with unit normal outward normal $\vn$ and we assume that the object is homogeneous with material coefficients $\epsilon_*, \mu_*$ and $\sigma_*$.
The object is surrounded by an unbounded region of free space $B_\alpha^c : = {\mathbb R}^3 \setminus \overline{B_\alpha}$ with material coefficients $\epsilon_0, \mu_0$ and $\sigma=0$ so that, henceforth,
\begin{align}
\tilde{\epsilon}_r (\vx)  := \left \{ \begin{array}{ll} \epsilon_r :=  \frac{1}{\epsilon_0} \left ( \epsilon_* + \frac{\im \sigma_*}{\omega } \right )  & \vx \in B_\alpha  \\
1 & \vx\in B_\alpha^c \end{array} \right . ,
\qquad \tilde{\mu}_r  :=\left \{ \begin{array}{ll} \mu_r : = \frac{\mu_*}{\mu_0}  & \vx \in B_\alpha  \\
1 & \vx \in B_\alpha^c \end{array} \right . . \nonumber
\end{align}
Furthermore, we assume that $\vJ$ has support only in $B_\alpha^c$ and for it to be located away from $B_\alpha$.
The electric and magnetic fields obey the standard transmission conditions
\begin{subequations}\label{eqn:transcond}
\begin{align} 
[\vn  \times \vEa]_{\Gamma_\alpha} & = \vzero , \\
[\vn \times \vHa]_{\Gamma_\alpha} & = \vzero  , \\
[\vn \cdot\tilde{\epsilon}_r \vEa]_{\Gamma_\alpha} & = 0 , \\
[\vn \cdot \tilde{\mu}_r \vHa]_{\Gamma_\alpha}  & = 0  ,
\end{align}
\end{subequations}
on $\Gamma_\alpha:= \partial B_\alpha$ where $[\cdot]_{\Gamma_\alpha} = \cdot|_+ -\cdot |_-$ denotes the jump and $+/-$  the evaluation just outside/inside of $B_\alpha$, respectively. In absence of $B_\alpha$, the background (incident) fields $\vEb$ and  $\vHb$ satisfy (\ref{eqn:maxwell}) with $\tilde{\epsilon}_r = \tilde{\mu}_r=1$ in ${\mathbb R}^3 $. The fields 
$\vEd:=\vEa-\vEb$ and $\vHd:=\vHa-\vHb$, which represent the perturbation in the electric and magnetic fields due to the presence of $B_\alpha$, respectively, satisfy the radiation conditions
\begin{align}
\lim_{x\to \infty} x \left ( ( \nabla \times \vEd) \times \hat{\vx} - \im k  \vEd \right ) =&  \vzero, \nonumber\\
\lim_{x\to \infty}  x\left ( ( \nabla \times \vHd)  \times \hat{\vx} - \im k  \vHd \right ) =&  \vzero, \nonumber
\end{align}
where $x=|\vx|$ and $\hat{\vx} = \vx / x$. By eliminating $\vHa$, we arrive at the following transmission problem 
\begin{subequations} \label{eqn:transeproblem}
\begin{align} 
\nabla \times {\mu}_r^{-1} \nabla \times \vEd -k^2 \epsilon_r \vEd &= \nabla\times ( 1- \mu_r^{-1}) \nabla \times \vEb -k^2(1-\epsilon_r ) \vEb &&  {} \text{in $B_\alpha$}, \\
\nabla \times \nabla \times \vEd -k^2  \vEd &= \vzero &&\text{in $B_\alpha^c$}, \\
\nabla \cdot ( \tilde{\epsilon}_r \vEd) & = 0 && \text{in ${\mathbb R}^3$}, \\
[\vn \times \vEd]_{\Gamma_\alpha} & = \vzero &&\text{on $\Gamma_\alpha$}, \\
[\vn \times \tilde{\mu}_r^{-1} \nabla \times \vEd]_{\Gamma_\alpha}  & = -[\tilde{\mu}_r^{-1}]_{\Gamma_\alpha} \vn \times \nabla  \times \vEb &&\text{on $\Gamma_\alpha$}, \\
\lim_{x\to \infty} x \left (  ( \nabla \times \vEd) \times \hat{\vx} - \im k  \vEd \right ) &=  \vzero,
\end{align} 
\end{subequations}
for $\vEd$, while an analogous transmission problem could also be derived for $\vHd$. Given the solenoidal behaviour of $\vEb$,  the condition $\nabla \cdot ( \tilde{\epsilon}_r \vEd)  = 0$ in ${\mathbb R}^3$ can be disregarded provided that $k^2>0$ is not an eigenvalue of (\ref{eqn:transeproblem}), which we shall assume throughout.

Introducing the function space
\begin{equation}
\vX := \left \{ \vu: \nabla \times \vu \in (L^2({\mathbb R}^3))^3, \vu \in (L^2({\mathbb R}^3))^3, \lim_{x\to \infty} x \left ( ( \nabla \times \vu) \times \hat{\vx} - \im k  \vu \right )=\vzero \right \} \nonumber ,
\end{equation}
 the weak form of (\ref{eqn:transeproblem}) is: Find $\vEd \in\vX$ s.t.
\begin{align}
\int_{B_\alpha^c} \left (  \nabla \times  \vEd \cdot \nabla \times \overline{\vv} - k^2 \vEd\cdot \overline{\vv} \right ) \dif \vx + &  \int_{B_\alpha} \left (  \mu_r^{-1} \nabla \times  \vEd \cdot \nabla \times \overline{ \vv} -k^2 \epsilon_r  \vEd\cdot \overline{\vv}  \right ) \dif \vx \nonumber\\
= & \int_{B_\alpha} \left ( (1-  \mu_r^{-1}) \nabla \times \vEb \cdot \nabla \times \overline{ \vv} -k^2 (1-  \epsilon_r)  \vEb\cdot \overline{\vv}  \right ) \dif \vx,
\end{align}
for all $\vv \in \vX$.  

We also recall that
\begin{equation}
G_k(\vx,\vy) := \frac{e^{\im k | \vx-\vy|}}{4\pi |\vx-\vy|} , \nonumber
\end{equation}
is the Helmholtz's free space Green's function, which becomes the Laplace free space Green's function if $k=0$.

\section{Main result} \label{sect:main}

Our goal is to derive an asymptotic formula for $\vHd(\vx) = (\vHa-\vHb)(\vx)$ for $\vx$ away from $B_\alpha$ as $\alpha\to 0$. We are interested in the regime for which  {$ \nu = O(1)$  as $\alpha \to 0$} where
\begin{equation}
\nu=\alpha^2 k^2 (\epsilon_r-1) = \nu_{\rm r} + \im \nu_\im ,\qquad  \nu_{\rm r} =(\epsilon_*- \epsilon_0) \mu_0 \omega^2 \alpha^2, \qquad \nu_\im = \sigma_*\mu_0 \omega \alpha^2, \nonumber
\end{equation}
and  $\mu_r =O(1)$.   

Our treatment will allow an  extension of the results described in~\cite{Ammari2014,LedgerLionheart2015,LedgerLionheart2018g}, which considered the eddy current model,   {where the quasi-static assumption (i.e. $\alpha \ll \lambda_0= 2\pi / k$) and large conductivities ($\sigma_* \gg \epsilon_* \omega$) were assumed, and, instead, considered the regime where $ \nu_\im =O(1)$ and $\mu_r =O(1)$ as $\alpha \to 0$.  Fixing $\nu=O(1)$, rather than $ \nu_\im =O(1)$, means that, in addition to including the case where $\sigma_*$ and $\omega $ are constant, we  include the situation where $\epsilon_*$ is constant. These quantities are also allowed to decrease with $\alpha$, but not faster than $1/\alpha^2$.  This allows us to consider how the characterisation of $B_\alpha$ changes for frequencies in the regime where neither the displacement nor Ohmic currents dominate.}
 We focus on  $\vHd(\vx)  = (\vHa-\vHb)(\vx)$ to allow ease of comparison with these earlier results. %, although similar steps could be followed  for $\vEd(\vx) =  (\vEa-\vEb)(\vx)$. 
 Our main result is
\begin{theorem} \label{thm:main}
For $\vx$ away from $B_\alpha = \alpha B + \vz$, the following expansion of $\vHd (\vx) = (\vHa-\vHb)(\vx)$ holds
\begin{align}
(\vHd(\vx))_j = & -\im k (\nabla_x G_k(\vx,\vz))_p \varepsilon_{jpr} ( {\mathcal A}_{ri} ( \vHb(\vz))_i + {\mathcal B}_{ri}(\vEb(\vz))_i )\nonumber \\
& + (\vD_x^2 G_k(\vx,\vz))_{\ell m} \varepsilon_{j\ell s} {\mathcal C}_{msi} ( \vHb(\vz))_i \nonumber \\
& +( (\vD_x^2 G_k(\vx,\vz))_{j r}+ k^2 \delta_{jr} G_k(\vx,\vz)) {\mathcal N}_{ri} ( \vHb(\vz))_i + (\vR(\vx))_j , \label{eqn:main}
\end{align}
where ${\mathcal A}, {\mathcal B}, {\mathcal C}$ and ${\mathcal N}$ are polarizability tensors with coefficients
\begin{subequations}
\begin{align}
{\mathcal A}_{ri}   := & \frac{ \im k \alpha^4 (\epsilon_r -1)}{2} \ve_r \cdot  \int_{B}  \left (  \ve_i \times \vxi + \vth_i \right ) \dif \vxi , \\
{\mathcal B}_{ri}   := & \alpha^3  (\epsilon_r -1) \ve_r\cdot  \int_B \left ( \ve_i   + \vphi_i \right ) \dif \vxi ,  \\
{\mathcal C}_{msi} :=& -\frac{k^2 \alpha^5 (\epsilon_r -1)}{2} \ve_s \cdot \int_B \xi_m  \left (  \ve_i \times \vxi + \vth_i \right ) \dif \vxi   , \\
{\mathcal N}_{ri}   := &  \alpha^3 (1-\mu_r^{-1}) \ve_r \cdot \int_B \left ( \ve_i + \frac{1}{2} \nabla \times \vth_i \right ) \dif \vxi,
\end{align}
\end{subequations}
that depend on the solution to the transmission problems
\begin{subequations} 
\begin{align}
\nabla_\xi \times \mu_r^{-1} \nabla_\xi \times \vth_i -k^2 \alpha^2  \epsilon_r  \vth_i=  &  -k^2\alpha^2 (1-\epsilon_r)\ve_i\times \vxi && \text{in $B$},\\
\nabla_\xi \times \nabla_\xi  \times\vth_i -k^2\alpha^2  \vth_i = & \vzero && \text{in $B^c:={\mathbb R}^3 \setminus \overline{B}$},\\
\nabla_\xi\cdot \vth_i   = & 0 && \text{in ${\mathbb R}^3$},\\
[\vn \times\vth_i ]_\Gamma=\vzero, \qquad [\vn \times \tilde{\mu}_r^{-1} \nabla_\xi \times\vth_i ]_\Gamma =& -2 (1-\mu_r^{-1})\vn \times \ve_i && \text{on $\Gamma:= \partial B$},\\
\lim_{\xi \to \infty}\xi \left ( ( \nabla_\xi \times \vth_i ) \times  \hat{\vxi} - \im k  \alpha\vth_i \right ) = &  \vzero,
\end{align}
\end{subequations}
and 
\begin{subequations} 
\begin{align}
\nabla_\xi \times \mu_r^{-1} \nabla_\xi \times \vphi_{i}  -k^2 \alpha^2 \epsilon_r \vphi_{i} =  &  -k^2\alpha^2 (1-\epsilon_r)
\ve_i  && \text{in $B$} ,\\
\nabla_\xi \times \nabla_\xi  \times\vphi_{i} -k^2\alpha^2 \vphi_{i}= & \vzero && \text{in $B^c$},\\
\nabla_\xi\cdot \vphi_{i}   = & 0 && \text{in ${\mathbb R}^3$},\\
[\vn \times\vphi_{i} ]_\Gamma =\vzero, \qquad [\vn \times \tilde{\mu}_r^{-1} \nabla_\xi \times\vphi_{i} ]_\Gamma  =& \vzero&& \text{on $\Gamma$},\\
\lim_{\xi \to \infty} \xi \left ( ( \nabla_\xi \times\vphi_{i} ) \times \hat{\vxi} - \im k \alpha\vphi_{i} \right ) = &  \vzero.
\end{align}
\end{subequations}
The residual $\vR(\vx)$  satisfies 
$ |\vR(\vx)| \le C \left ( \alpha^4  \| \vHb \|_{W^{2,\infty}(B_\alpha)} + \alpha^4 k \left (  |\epsilon_r-1| + 
\alpha k^2   \left  |1 -\frac{1}{\epsilon_r} \right | \right )  \| \vEb \|_{W^{1,\infty}(B_\alpha)} \right )$.  \end{theorem}

The proof of this theorem follows in Section~\ref{sect:proof}. Beforehand,  we derive results that our proof will draw on.

\section{Energy estimates} \label{sect:energy}

To arrive at an energy estimate for an eddy current problem, 
Ammari, Chen, Chen, Garnier and Volkov~\cite{Ammari2014} have introduced the following 
\begin{align}
\vF(\vx) := &  \frac{1}{2} (\nabla_z \times \vEb(\vz))\times(\vx-\vz) + \frac{1}{3} \vD_z (\nabla_z \times \vEb(\vz))(\vx-\vz) \times(\vx-\vz)  \nonumber \\
= &   \frac{\im k }{2} \sum_{i=1}^3  (\vHb(\vz))_i \ve_i  \times (\vx-\vz) + \frac{\im k}{3} \sum_{i,j=1}^3 \vD_z (\vHb(\vz) )_{ij} (\vx-\vz)_j \ve_i \times(\vx-\vz) ,\label{eqn:oldf}
\end{align}
which has the curl
\begin{align}
\nabla_x \times \vF  = &  \nabla_z \times \vEb(\vz) + \vD_z (\nabla_z \times \vEb(\vz))(\vx-\vz) \nonumber \\
= & \im k \vHb(\vz)  +\im k \sum_{i,j=1}^3 \vD_z (\vHb(\vz))_{ij} (\vx -\vz )_j \ve_i . \label{old:curlf}
\end{align}
In the above, $\nabla_x \times \vF$ corresponds to the first two terms in a Taylor's series expansion of $\im k \vHb(\vx)$ about $\vz$ as $|\vx-\vz| \to 0$. We adopt a different form for $\vF$ as follows
\begin{align}
\vF(\vx) := &   \frac{1}{2} (\nabla_z \times \vEb(\vz))\times(\vx-\vz) + \frac{1}{3} \vD_z (\nabla_z \times \vEb(\vz))(\vx-\vz) \times(\vx-\vz)  + \vEb(\vz)  \nonumber \\
= &   \frac{\im k }{2} \sum_{i=1}^3  (\vHb(\vz))_i \ve_i  \times (\vx-\vz) + \frac{\im k}{3} \sum_{i,j=1}^3 \vD_z (\vHb(\vz) )_{ij} (\vx-\vz)_j \ve_i \times(\vx-\vz)  + \vEb(\vz),\label{eqn:newf}
\end{align}
so that
\begin{align}
\nabla_x \times \vF  =   \im k  \vHb(\vz)  + \im k \sum_{i,j=1}^3 \vD_z (\vHb(\vz))_{ij} (\vx -\vz )_j \ve_i  , \label{eqn:newcurlf}
\end{align}
is still the first two terms in a Taylor's series expansion of $\im k \vHb(\vx)$ about $\vz$ as $|\vx-\vz| \to 0$.
Note that $\vF(\vz) = \vE_0(\vz)$ and $(\nabla_x \times  \vF)(\vz) = \im k \vHb(\vz)$. We can also determine
\begin{align}
\nabla_x \cdot \vF  = & -   \frac{\im k}{3} \sum_{i=1}^3 (  \nabla_z \times \vHb(\vz))_{i}\ve_i \cdot (\vx -\vz ) \nonumber \\
=&    \frac{ k^2}{3} \sum_{i=1}^3 ( \vEb(\vz))_{i}\ve_i \cdot (\vx -\vz ) \nonumber \\
= &   \nabla^2 \phi_0,
\end{align}
where, up to a constant, 
\begin{align}
 \phi_0 = \frac{k^2}{18}  \left ( (\vEb(\vz))_1 (x_1-z_1)^3 +  (\vEb(\vz))_2 (x_2-z_2)^3 +  (\vEb(\vz))_3 (x_3-z_3)^3 \right ).\label{eqn:phi}
\end{align}

Related to (3.7) in~\cite{Ammari2014}, we introduce $\vw$ as the  {unique} solution to the weak problem: Find $\vw \in \vX$ s.t.
\begin{align}
\int_{B_\alpha^c} \left (  \nabla \times  \vw \cdot \nabla \times \overline{\vv} - k^2 \vw\cdot \overline{\vv} \right ) \dif \vx + &  \int_{B_\alpha} \left (  \mu_r^{-1} \nabla \times  \vw \cdot \nabla \times \overline{ \vv} -k^2 \epsilon_r  \vw \cdot \overline{\vv}  \right ) \dif \vx \nonumber\\
= & \int_{B_\alpha} \left ( (1-  \mu_r^{-1}) \nabla \times \vF \cdot \nabla \times \overline{\vv} -k^2 (1-  \epsilon_r)  \vF \cdot \overline{\vv}  \right ) \dif \vx \label{eqn:weakw} ,
\end{align}
for all $\vv \in \vX$.

The following updated form of Lemma 3.2 from~\cite{Ammari2014}  relies on a Friedrichs' type inequality that will allow us to estimate a divergence free field $\vu\in H(\text{curl})\cap H(\text{div})$ in a bounded Lipschitz domain $B_\alpha$ in terms of its curl and $\vn \cdot \vu$ on $\partial B_\alpha$ e.g.~\cite{monk}[Corollary 3.52, pg 72]
\begin{align}
\| \vu \|_{L^2(B_\alpha)} \le C  \left ( \| \nabla \times \vu \|_{L^2(B_\alpha)} + \| \vn \cdot \vu \|_{L^2(\Gamma_\alpha )} \right ), \label{eqn:fried}
\end{align}
which we assume holds with $\alpha=1$ while,  {in general},
\begin{align}
\| \vu \|_{L^2(B_\alpha)} \le C \alpha \left ( \| \nabla \times \vu \|_{L^2(B_\alpha )} + \| \vn \cdot \vu \|_{L^2(\Gamma_\alpha )} \right ), \label{eqn:frieddim}
\end{align}
leading to:
\begin{lemma} \label{lemma:energy}
Let $\vw$ be defined by (\ref{eqn:weakw}). For $|\epsilon_r|$ such that 
$ \frac{1}{|\epsilon_r|} \| \vn \cdot ( \vEd|_+- \vw|_+)\|_{L^2(  \Gamma_\alpha) } \le 
  \| \nabla \times ( \vEa - \vEb-   \vw)   \|_{L^2(B_\alpha)}$,
   there exists a constant $C$ such that
\begin{align}
\| \nabla \times ( \vEa - \vEb - \vw ) \|_{L^2(B_\alpha)}  \le &  C \alpha^{7/2} \mu_r  \left (|1 - \mu_r^{-1}| + | \nu | \right ) \| \nabla \times \vEb\|_{W^{2,\infty}(B_\alpha)} \nonumber , \\
\left \| \vEa - \vEb - \vw - \frac{\epsilon_r-1}{\epsilon_r} \nabla \phi_0  \right  \|_{L^2(B_\alpha)}  \le &  C \alpha^{9/2} \mu_r  \left (|1 - \mu_r^{-1}| + | \nu | \right ) \| \nabla \times \vEb\|_{W^{2,\infty}(B_\alpha)} \nonumber . 
\end{align}
\end{lemma}
\begin{proof}
For the eddy current model considered by~\cite{Ammari2014},  $\nabla \cdot \vEa = \nabla \cdot \vEb =0$ in $B_\alpha$ and $\vn \cdot \vEa|_-=0$ on $\Gamma_\alpha$. Also,  {they have} $\nabla\cdot \vw  = - \nabla \cdot \vF $ in $B_\alpha$ and $\vn \cdot \vw|_- = - \vn \cdot \vF $ on $\Gamma_\alpha$. Our situation is different, however,
so that choosing $\vv= \nabla \vartheta$ for some appropriate $\vartheta$ in (\ref{eqn:weakw}) we find that
\begin{align}
\nabla \cdot (\epsilon_r \vw ) = \nabla \cdot ((1-\epsilon_r) \vF) \qquad \text{in $B_\alpha$}, \qquad \vn \cdot \vw|_+ - \vn \cdot \epsilon_r \vw|_- = -   (1-\epsilon_r) \vn \cdot \vF \qquad \text{on $\Gamma_\alpha$} \nonumber.
\end{align}
We also know that for the problem we are considering
\begin{align}
\nabla \cdot (\epsilon_r \vEa ) = 0 \qquad \text{in $B_\alpha$}, \qquad \vn \cdot \vEa |_+ - \vn \cdot \epsilon_r \vEa |_- = 0 \qquad \text{on $\Gamma_\alpha$}  \nonumber.
\end{align}
We follow~\cite{Ammari2014}, and introduce $\phi_0$ in a similar way such that $\nabla \cdot  ( \vEb + \nabla \phi_0) = \nabla \cdot \vF$ in $B_\alpha$ with $\vn \cdot (\vEb +\nabla \phi_0) |_- = \vn \cdot \vF$ on $\Gamma_\alpha$, which is justified based on  (\ref{eqn:phi}) and knowing that $\nabla \cdot \vEb=0$ in $B_\alpha$.  Since $\nabla \cdot ( \vF - \nabla \phi_0) =0$ in $B_\alpha$ then $\nabla \cdot ((1-\epsilon_r ) (\vF - \nabla \phi_0)) = \nabla \cdot (\epsilon_r \vw - (1-\epsilon_r) \nabla \phi_0) =0$ and $\nabla \cdot ( \vw -\frac{1-\epsilon_r}{\epsilon_r} \nabla \phi_0) =0$ in $B_\alpha$. Then,
\begin{align}
\nabla \cdot \left ( \vEa - \vEb - \vw - \frac{\epsilon_r-1}{\epsilon_r} \nabla \phi_0 \right ) =0 \qquad \text{in $B_\alpha$},
\end{align}
and, by combining the above results, we also have
\begin{align}
\vn \cdot \left  ( \vEa|_- -\vw|_- -\vEb - \frac{\epsilon_r-1}{\epsilon_r}\nabla \phi_0  \right ) -\frac{1}{\epsilon_r}  \vn \cdot \left ( \vEa|_+ - \vw|_+ - \vEb  \right ) =0 \qquad \text{on $\Gamma_\alpha$}.
\end{align}
  {For  $|\epsilon_r| \to \infty$,  then $ \frac{1}{\epsilon_r}  \vn \cdot \left ( \vEa|_+ - \vw|_+ - \vEb  \right )=  \frac{1}{\epsilon_r}  \vn \cdot \left ( \vEd|_+ - \vw|_+   \right ) \to 0$ on $\Gamma_\alpha$ and, the presence of this term,} means that  rather than \begin{align}
\left \| \vEa - \vEb-\vw -\frac{\epsilon_r-1}{\epsilon_r }\nabla \phi_0  \right \|_{L^2(B_\alpha)} \le C \alpha \| \nabla \times ( \vEa - \vEb-   \vw)   \|_{L^2(\Gamma_\alpha)},  \label{eqn:eqnlemmacon4old}
\end{align}
we should use (\ref{eqn:fried}) leading to an estimate of the form
\begin{align}
\left \| \vEa - \vEb-\vw -\frac{\epsilon_r-1}{\epsilon_r }\nabla \phi_0  \right \|_{L^2(B_\alpha)} \le C \alpha \left (  \| \nabla \times ( \vEa - \vEb-   \vw)   \|_{L^2(B_\alpha)}+ \frac{1}{|\epsilon_r|} \| \vn \cdot ( \vEd|_+- \vw|_+)\|_{L^2(  \Gamma_\alpha) }  \right )  .  \label{eqn:eqnlemmacon4}
\end{align}
  
Next, we construct,
\begin{align}
\int_{B_\alpha^c} & \left (  \nabla \times (\vEa - \vEb  - \vw  ) \cdot \nabla \times \overline{\vv} - k^2 (\vEa - \vEb  - \vw  )\cdot \overline{\vv} \right ) \dif \vx \nonumber\\
&+ \int_{B_\alpha} \left (  \mu_r^{-1} \nabla \times (\vEa -   \vEb  -  \vw  ) \cdot \nabla \times \overline{\vv} - k^2 \epsilon_r \left  (\vEa - \vEb  -\vw-\frac{\epsilon_r-1}{\epsilon_r} \nabla \phi_0 \right  )\cdot \overline{\vv} \right ) \dif \vx \nonumber\\
& = \int_{B_\alpha}  \left ( (1-\mu_r^{-1})\nabla \times (\vEb-  \vF) \cdot \nabla \times \overline{\vv} - k^2  (1-\epsilon_r) (  \vEb+  \nabla \phi_0  -\vF)  ) \cdot \overline{\vv} \right ) \dif \vx. \label{eqn:eqnlemmacon1a}
\end{align}
By noting the  curl of the gradient of a scalar vanishes we can write
\begin{align}
\int_{B_\alpha^c} & \left (  \nabla \times (\vEa - \vEb  - \vw ) \cdot \nabla \times \overline{\vv} - k^2 (\vEa - \vEb  - \vw  )\cdot \overline{\vv} \right ) \dif \vx \nonumber\\
&+ \int_{B_\alpha} \left (  \mu_r^{-1} \nabla \times \left (\vEa -   \vEb  -  \vw - \frac{\epsilon_r-1}{\epsilon_r} \nabla \phi_0 \right  ) \cdot \nabla \times \overline{\vv} - k^2 \epsilon_r \left (\vEa - \vEb  -\vw-\frac{\epsilon_r-1}{\epsilon_r} \nabla \phi_0 \right  )\cdot \overline{\vv} \right ) \dif \vx \nonumber\\
& = \int_{B_\alpha}  \left ( (1-\mu_r^{-1})\nabla \times (\vEb-  \vF) \cdot \nabla \times \overline{\vv} - k^2  (1-\epsilon_r) (  \vEb+  \nabla \phi_0  -\vF)  ) \cdot \overline{\vv} \right ) \dif \vx. \label{eqn:eqnlemmacon1}
\end{align}
Choosing $\vv = \left \{ \begin{array}{lc} \vEa -   \vEb  -  \vw - \frac{\epsilon_r-1}{\epsilon_r} \nabla \phi_0 & \text{in $B_\alpha$} \\
\vEa -   \vEb  -  \vw & \text{in $B_\alpha^c $} \end{array} \right .$
we have 
\begin{align}
& \left  | \int_{B_\alpha}  \mu_r^{-1}  \left   |  \nabla \times (\vEa - \vEb -\vw-\frac{\epsilon_r-1}{\epsilon_r} \nabla \phi_0 \right   |^2 \dif \vx \right |  \nonumber\\
&\qquad  \le \left | \int_{B_\alpha^c}  | \nabla \times (\vEa - \vEb -\vw  ) |^2 - k^2 | \vEa - \vEb -\vw  |^2  \dif \vx\right . \nonumber\\
&\qquad \qquad \left .+ \int_{B_\alpha} \left (  \mu_r^{-1} \left   | \nabla \times (\vEa - \vEb -\vw-\frac{\epsilon_r-1}{\epsilon_r} \nabla \phi_0   \right |^2  - k^2 \epsilon_r \left | \vEa - \vEb -\vw-\frac{\epsilon_r-1}{\epsilon_r} \nabla \phi_0   \right |^2  \right ) \dif \vx \right | \nonumber\\
& \qquad = \left | \int_{B_\alpha}  \left ( (1-\mu_r^{-1})\nabla \times (\vEb-\vF) \cdot \nabla \times \overline{\vv} - k^2 (1-\epsilon_r) (\vEb+\nabla \phi_0  -\vF  ) \cdot \overline{\vv} \right ) \dif \vx \right | \label{eqn:eqnlemmacon1n}.
\end{align}
Also,
\begin{align}
\left |  \int_{B_\alpha} (1-\mu_r^{-1}) \left ( \nabla \times (\vEb-\vF) \cdot \nabla \times \overline{\vv} \right ) \dif \vx \right | \le
C | 1-\mu_r^{-1} | \alpha^{7/2}\| \nabla \times \vEb \|_{W^{2,\infty}(B_\alpha)} \| \nabla \times \vv \|_{L^2(B_\alpha)} , \label{eqn:eqnlemmacon2}
\end{align}
and
\begin{align}
\left | k^2  \int_{B_\alpha} (1-\epsilon_r) (\vEb+\nabla \phi_0 -\vF  )\cdot \overline{\vv}  \dif \vx \right | \le  & k^2 |1-\epsilon_r | \| \vEb+ \nabla \phi_0 -\vF   \|_{L^2(B_\alpha)} \| \vv \|_{L^2(B_\alpha)} \nonumber \\
\le & C\alpha^2 k^2 |1-\epsilon_r | \| \nabla \times ( \vEb-\vF ) \|_{L^2(B_\alpha)} ( \| \nabla \times \vv \|_{L^2(B_\alpha)} +\| \vn \cdot \vv \|_{L^2(\Gamma_\alpha)}  ) \nonumber \\
\le & C  \alpha^{7/2} | \nu |   \| \nabla \times  \vEb  \|_{W^{2,\infty} (B_\alpha)}(  \| \nabla \times \vv \|_{L^2(B_\alpha)} +\| \vn \cdot \vv \|_{L^2(\Gamma_\alpha)})  , \label{eqn:eqnlemmacon3}
\end{align}
which follows since $\vEb+\nabla \phi_0-\vF  $ is divergence free in $B_\alpha$ and has vanishing  {$\vn \cdot( \vEb+\nabla \phi_0-\vF)$} on $\Gamma_\alpha$ and so
\begin{align}
\| \vEb+\nabla \phi_0 -\vF   \|_{L^2(B_\alpha)} \le C\alpha \| \nabla \times ( \vEb-\vF ) \|_{L^2(B_\alpha)} \le & C\alpha \alpha^{3/2} \alpha^2 \| \nabla \times  \vEb  \|_{W^{2,\infty}(B_\alpha)}\nonumber \\
 =&  C\alpha^{9/2}  \| \nabla \times  \vEb  \|_{W^{2,\infty}(B_\alpha)} \label{eqn:ineqebf}.
\end{align}
Also, since $\vv$ is divergence free, $\| \vv \|_{L^2(B_\alpha)} \le C \alpha \left (  \| \nabla \times \vv \|_{L^2(B_\alpha)} + \| \vn \cdot \vv \|_{L^2(\Gamma_\alpha)} \right )  $.
Using (\ref{eqn:eqnlemmacon2}) and (\ref{eqn:eqnlemmacon3})  in (\ref{eqn:eqnlemmacon1n}) then 
\begin{align}
& \mu_r^{-1} \left \| \nabla \times \left (  \vEa - \vEb-\frac{\epsilon_r-1}{\epsilon_r} \nabla \phi_0 -\vw \right ) \right  \|_{L^2(B_\alpha)}^2 \le
C  | 1-\mu_r^{-1} | \alpha^{7/2}\| \nabla \times \vEb \|_{W^{2,\infty}(B_\alpha)} \| \nabla \times \vv \|_{L^2(B_\alpha)}  \nonumber\\
&\qquad \qquad \qquad \qquad \qquad \qquad \qquad \qquad +
C \alpha^{7/2} | \nu |   \| \nabla \times  \vEb  \|_{W^{2,\infty} (B_\alpha)} (\| \nabla \times \vv \|_{L^2(B_\alpha)}+\| \vn \cdot \vv \|_{L^2(\Gamma_\alpha)}) ,
\end{align}
and for  {$|\epsilon_r|$  such that 
$\| \vn \cdot \vv \|_{L^2(\Gamma_\alpha)} = \frac{1}{|\epsilon_r|} \| \vn \cdot ( \vEd|_+- \vw|_+)\|_{L^2(  \Gamma_\alpha) } \le 
  \| \nabla \times ( \vEa - \vEb-   \vw)   \|_{L^2(B_\alpha)}$,} then
\begin{align}
 \| \nabla \times (  \vEa - \vEb -\vw )\|_{L^2(B_\alpha)} \le&
C \mu_r \alpha^{7/2}\left (  | 1-\mu_r^{-1} | + |  \nu | \right ) \| \nabla \times \vEb \|_{W^{2,\infty}(B_\alpha)} .
\end{align}
By additionally combining this with (\ref{eqn:eqnlemmacon4}) completes the proof.

\end{proof}

Unlike the eddy current problem considered in~\cite{Ammari2014}, we are not guaranteed to have $\nabla_x \times \vHb(\vx)=\vzero$ in $B_\alpha$ and so an integration by parts yields
\begin{align}
\int_{B_\alpha} (1-\mu_r^{-1}) \nabla \times \vF\cdot \nabla \times \overline{\vv}  \dif \vx = & \int_{\Gamma_\alpha} [ \tilde{\mu}_r^{-1} \nabla \times \vF \times \vn^-]_{\Gamma_\alpha} \cdot \overline{\vv} \dif \vx + \int_{B_{\alpha}} (1-\mu_r^{-1}) \overline{\vv} \cdot \nabla \times \nabla \times \vF \dif \vx  \nonumber \\
\end{align}
and, hence, the weak problem for $\vw$ becomes:
Find $\vw \in \vX$ s.t.
\begin{align}
\int_{B_\alpha^c} \left (  \nabla \times  \vw \cdot \nabla \times \overline{\vv} - k^2 \vw\cdot \overline{\vv} \right ) \dif \vx + &  \int_{B_\alpha} \left (  \mu_r^{-1} \nabla \times  \vw \cdot \nabla \times \overline{ \vv} -k^2 \epsilon_r  \vw \cdot \overline{\vv}  \right ) \dif \vx \nonumber\\
= &\int_{\Gamma_\alpha} [ \tilde{\mu}_r^{-1} \nabla \times \vF \times \vn^-]_{\Gamma_\alpha} \cdot \overline{\vv} \dif \vx -k^2\int_{B_\alpha} \left ( (1-  \epsilon_r)  \vF \cdot \overline{\vv}  \right ) \dif \vx \nonumber \\
 & + \int_{B_{\alpha}} (1-\mu_r^{-1}) \overline{\vv} \cdot \nabla \times \nabla \times \vF \dif \vx , \label{eqn:weakw2}
\end{align}
for all $\vv \in \vX$. This, in turn, motivates the strong form for $\vw$ as
\begin{subequations}
\begin{align}
\nabla \times \mu_r^{-1} \nabla \times \vw -k^2 \epsilon_r \vw =  &   (1-\mu_r^{-1}) \nabla \times \nabla \times \vF -k^2(1-\epsilon_r) \vF && \text{in $B_\alpha$},\\
\nabla \times \nabla \times \vw -k^2  \vw = & \vzero && \text{in $B_\alpha^c$},\\
\nabla\cdot \vw  = & 0 && \text{in ${\mathbb R}^3$},\\
[\vn \times \vw]_{\Gamma_\alpha}=\vzero, \qquad [\vn \times \tilde{\mu_r}^{-1} \nabla \times \vw]_{\Gamma_\alpha} =& -(1-\mu_r^{-1})\vn \times \nabla \times \vF && \text{on $\Gamma_\alpha$},\\
\lim_{x\to \infty} x \left ( ( \nabla \times \vw) \times \hat{\vx}- \im k  \vw \right ) = &  \vzero.
\end{align}
\end{subequations}
By introducing $\vw (\vx) = \alpha \vwz\left ( \frac{\vx-\vz}{\alpha} \right ) = \alpha \vwz(\vxi) $, we see $\vw_0(\vxi)$ satisfies
\begin{subequations}
\begin{align}
\nabla_\xi \times \mu_r^{-1} \nabla_\xi \times \vwz -k^2 \epsilon_r\alpha^2  \vwz =  &  (1-\mu_r^{-1}) \nabla \times \nabla \times [\alpha \vF]  -k^2\alpha^2 (1-\epsilon_r) [\alpha^{-1} \vF] && \text{in $B$},\\
\nabla_\xi \times \nabla_\xi  \times \vwz -k^2 \alpha^2 \vwz = & \vzero && \text{in $B^c$},\\
\nabla_\xi\cdot \vwz  = & 0 && \text{in ${\mathbb R}^3$},\\
[\vn \times \vwz]_\Gamma=\vzero, \qquad [\vn \times \tilde{\mu}_r^{-1} \nabla_\xi \times \vwz]_\Gamma =& -(1-\mu_r^{-1})\vn \times \nabla \times [\alpha^{-1} \vF] && \text{on $\Gamma$},\\
\lim_{\xi \to \infty}  \xi \left (  ( \nabla_\xi \times \vwz) \times  \hat{\vxi} - \im k  \alpha \vwz \right ) = &  \vzero.
\end{align}\label{eqn:transw0}
\end{subequations}

Hence, Theorem 3.1 in~\cite{Ammari2014}, which follows directly from Lemma~\ref{lemma:energy}, is replaced by
\begin{theorem} \label{thm:energyestf}
 {For $|\epsilon_r|$ such that 
$ \frac{1}{|\epsilon_r|} \| \vn \cdot ( \vEd|_+- \vw|_+)\|_{L^2(  \Gamma_\alpha) } \le 
  \| \nabla \times ( \vEa - \vEb-   \vw)   \|_{L^2(B_\alpha)}$,} there exists a constant $C$ such that
\begin{align}
\left \|  \nabla \times ( \vEa- \vEb - \alpha \vwz \left ( \frac{\vx-\vz}{\alpha}  \right ) \right \|_{L^2(B_\alpha)}  \le &  C \alpha^{7/2} \mu_r  \left (|1 - \mu_r^{-1}| +  | \nu | \right ) \| \nabla \times \vEb\|_{W^{2,\infty}(B_\alpha)}, \nonumber \\
\left \| \vEa - \vEb    - \alpha \vwz \left ( \frac{\vx-\vz}{\alpha}  \right )  - \frac{\epsilon_r-1}{\epsilon_r} \nabla \phi_0 \right
\|_{L^2(B_\alpha)}  \le &  C \alpha^{9/2} \mu_r  \left (|1 - \mu_r^{-1}| + | \nu | \right ) \| \nabla \times \vEb\|_{W^{2,\infty}(B_\alpha)}. \nonumber 
\end{align}
\end{theorem}

From (\ref{eqn:newf}),  we have
\begin{align}
\alpha^{-1} \vF ( \alpha \xi + \vz ) = & 
\frac{\im k }{2}\sum_{i=1}^3  ( \vHb(\vz))_i \ve_i \times \vxi + \frac{\im k \alpha }{3} \sum_{i,j=1}^3 
(\vD_z \vHb(\vz))_{ij}  \xi_j \ve_i \times \vxi  + \alpha^{-1} \sum_{i=1}^3 (\vEb(\vz))_i \ve_i
 , \nonumber
\end{align} 
and similarly
\begin{align}
 \vwz ( \vxi  ) =   \frac{\im k}{2} \sum_{i=1}^3 (\vHb(\vz))_i \vth_i(\vxi)  + \frac{\im k \alpha}{3}\sum_{i=1,j}^3
(\vD_z \vHb(\vz))_{ij}   \vpsi_{ij} (\vxi)  + \alpha^{-1} \sum_{i=1}^3 (\vEb(\vz))_i \vphi_i (\vxi)
 . \label{eqn:neww}
\end{align} 
In the above, $\vth_i(\vxi)$ solves
\begin{subequations} \label{eqn:transth}
\begin{align}
\nabla_\xi \times \mu_r^{-1} \nabla_\xi \times \vth_i -k^2 \alpha^2 \epsilon_r \vth_i=  &  -k^2\alpha^2 (1-\epsilon_r)\ve_i\times \vxi && \text{in $B$},\\
\nabla_\xi \times \nabla_\xi  \times\vth_i -k^2\alpha^2  \vth_i = & \vzero && \text{in $B^c$},\\
\nabla_\xi\cdot \vth_i   = & 0 && \text{in ${\mathbb R}^3$},\\
[\vn \times\vth_i ]_\Gamma=\vzero, \qquad [\vn \times \tilde{\mu}_r^{-1} \nabla_\xi \times\vth_i ]_\Gamma =& -2 (1-\mu_r^{-1})\vn \times \ve_i && \text{on $\Gamma$},\\
\lim_{\xi \to \infty} \xi \left ( ( \nabla_\xi \times \vth_i ) \times \hat{\vxi} - \im k  \alpha\vth_i \right ) = &  \vzero,
\end{align}
\end{subequations}
$\vpsi_{ij} (\vxi)$ solves
\begin{subequations} \label{eqn:transpsi}
\begin{align}
\nabla_\xi \times \mu_r^{-1} \nabla_\xi \times \vpsi_{ij} -k^2 \alpha^2 \epsilon_r \vpsi_{ij}=  & (1-\mu_r^{-1}) \ve_j \times \ve_i -k^2\alpha^2 (1-\epsilon_r) \xi_j \ve_i\times \vxi && \text{in $B$},\\
\nabla_\xi \times \nabla_\xi  \times\vpsi_{ij} -k^2\alpha^2  \vpsi_{ij} = & \vzero && \text{in $B^c$},\\
\nabla_\xi\cdot \vpsi_{ij}   = & 0 && \text{in ${\mathbb R}^3$},\\
[\vn \times\vpsi_{ij} ]_\Gamma=\vzero, \qquad [\vn \times \tilde{\mu}_r^{-1} \nabla_\xi \times\vth_i ]_\Gamma =& -3 (1-\mu_r^{-1})\vn \times \xi_j \ve_i && \text{on $\Gamma$},\\
\lim_{\xi \to \infty} \xi \left ( ( \nabla_\xi \times \vpsi_{ij} ) \times \hat{\vxi} - \im k  \alpha\vpsi_{ij} \right ) = &  \vzero,
\end{align}
\end{subequations}
and $\vphi_{i} (\vxi)$ solves
\begin{subequations} \label{eqn:transph}
\begin{align}
\nabla_\xi \times \mu_r^{-1} \nabla_\xi \times \vphi_{i}  -k^2 \alpha^2 \epsilon_r \vphi_{i} =  &  -k^2\alpha^2 (1-\epsilon_r)
\ve_i  && \text{in $B$} ,\\
\nabla_\xi \times \nabla_\xi  \times\vphi_{i} -k^2\alpha^2 \vphi_{i}= & \vzero && \text{in $B^c$},\\
\nabla_\xi\cdot \vphi_{i}   = & 0 && \text{in ${\mathbb R}^3$},\\
[\vn \times\vphi_{i} ]_\Gamma =\vzero, \qquad [\vn \times \tilde{\mu}_r^{-1} \nabla_\xi \times\vphi_{i} ]_\Gamma  =& \vzero&& \text{on $\Gamma$},\\
\lim_{\xi \to \infty} \xi \left ( ( \nabla_\xi \times\vphi_{i} ) \times \hat{\vxi} - \im k \alpha\vphi_{i} \right ) = &  \vzero.
\end{align}
\end{subequations}
Furthermore, given that the only source term is the gradient of a scalar, the solution to (\ref{eqn:transph}) can be expressed in terms of $\vphi_i = (\epsilon_r-1) \nabla  \vartheta_i$ where $\vartheta_i$ solves
\begin{subequations} \label{eqn:transphsca}
\begin{align}
\nabla_\xi \cdot  \epsilon_r\nabla_\xi  \vartheta_{i} =  &  0
  && \text{in $B$} ,\\
\nabla_\xi \cdot \nabla_\xi  \vartheta_{i}= &0 && \text{in $B^c$},\\
[\vartheta_i]_\Gamma = 0 \qquad \vn \cdot \nabla_\xi \vartheta_i|_+  - \vn \cdot \epsilon_r \nabla_\xi  \vartheta_i |_-  = & \vn \cdot \ve_i && \text{on $\Gamma$},\\
\vartheta_i  \to & 0 && \text{as $|\vxi|\to \infty$}.
\end{align}
\end{subequations}

\section{Integral representation formulae} \label{sect:integral}

A different integral representation formula for $\vHd(\vx ) = (\vHa-\vHb)(\vx) $ for $\vx$ away from $B_\alpha$ is required compared to that used in~\cite{Ammari2014}, since the eddy current approximation is no longer applied. The following is appropriate for describing the perturbed magnetic field outside of $B_\alpha$ and is the same as used in~\cite{LedgerLionheart2015pert}
\begin{equation}
\vHd (\vx) =\nabla_x \times \left ( \int_{\Gamma_\alpha} G_k(\vx,\vy) \vn \times \vHd(\vy) |_+ \dif \vy \right )  - \frac{\im}{k} \nabla_x \times \left (\nabla_x \times \int_{\Gamma_\alpha} G_k(\vx,\vy) \vn \times \vEd(\vy) |_+ \dif \vy  \right ). \label{eqn:intrepbd}
\end{equation}
We transform this result to be expressed in terms of volume integrals over $B_\alpha$ and express the result in the lemma below.

\begin{lemma}
An integral representation formula for $\vHd(\vx ) = (\vHa-\vHb)(\vx) $ for $\vx$ away from $B_\alpha$ expressed in terms of volume integrals over $B_\alpha$  is
\begin{align}
\vHd (\vx) =  &- \im k(\epsilon_r -1) \int_{B_\alpha} \nabla_x G_k (\vx,\vy) \times \vEa (\vy) \dif \vy + k^2 (\mu_r -1) \int_{B_\alpha} G_k(\vx,\vy) \vHa (\vy) \dif \vy \nonumber \\
&+ (\mu_r-1) \int_{B_\alpha} \vD_y^2 G_k(\vx,\vy) \vHa(\vy) \dif \vy .\label{eqn:intrepvol}
\end{align}

\end{lemma}
\begin{proof}
Using the transmission conditions (\ref{eqn:transcond}), transforming the first term of (\ref{eqn:intrepbd}) to a volume integral and then applying $\nabla \times(\va \times \vb) = \va \nabla\cdot \vb - \vb \nabla \cdot \va + ( \vb \cdot \nabla) \va - (\va \cdot \nabla)\vb$  gives
\begin{align}
\nabla_x \times & \left ( \int_{\Gamma_\alpha} G_k(\vx,\vy) \vn \times \vHd(\vy) |_+ \dif \vy \right ) =  \nabla_x \times \left ( \int_{B_\alpha} \nabla_y \times (G_k(\vx,\vy) \vHd(\vy)) \dif \vy \right ) \nonumber \\
 =& \nabla_x \times \left ( \int_{B_\alpha} \nabla_y G_k(\vx,\vy) \times \vHd(\vy) + G_k(\vx,\vy) \nabla_y \times  \vHd(\vy)  \dif \vy \right ) \nonumber \\
 =&   \int_{B_\alpha} \left ( \nabla_y G_k(\vx,\vy) (\nabla_x \cdot \vHd(\vy) )- \vHd(\vy) \nabla_x\cdot \nabla_y  G_k(\vx,\vy)  \right . \nonumber \\
& \left .+ (\vHd(\vy) \cdot \nabla_x) (\nabla_y G_k(\vx,\vy) )- (\nabla_y G_k(\vx,\vy) \cdot \nabla_x)( \vHd(\vy))  +  \nabla_x G_k(\vx,\vy) \times  \nabla_y \times  \vHd(\vy)  \dif \vy
 \right ) \dif \vy \nonumber \\
 =&   \int_{B_\alpha} \left (-k^2  \vHd(\vy) G_k(\vx,\vy) 
 - \vD_x^2 G_k(\vx,\vy)  (\vHd(\vy) )+   \nabla_x G_k(\vx,\vy) \times \nabla_y \times  \vHd(\vy)  \right ) \dif \vy \label{eqn:intpt1},
  \end{align}
where $\nabla_x\cdot \nabla_y  G_k(\vx,\vy)  = -\nabla_x\cdot \nabla_x  G_k(\vx,\vy) = k^2 G_k(\vx,\vy)$ and $\nabla_x \nabla_y  G_k(\vx,\vy) = - \vD_x^2 G_k (\vx,\vy)$ have been applied.
 Considering the  the second term in (\ref{eqn:intrepbd}),  and applying similar ideas, gives
\begin{align}
\nabla_x  &\times \left (\nabla_x \times \int_{\Gamma_\alpha} G_k(\vx,\vy) \vn \times \vEd(\vy) _+ \dif \vy  \right ) = \nabla_x \times \left (\nabla_x \times \int_{B_\alpha}  \nabla_y G_k(\vx,\vy) \times \vEd(\vy) + G_k(\vx,\vy) \nabla_y \times \vEd(\vy) \dif \vy \right ) \nonumber \\
& = \nabla_x \times \int_{B_\alpha} - k^2\vEd(\vy) G_k(\vx,\vy) - \vD_x^2  G_k(\vx,\vy) \vEd(\vy) + \nabla_x G_k(\vx,\vy) \times \nabla_y \times \vEd(\vy) \dif \vy \nonumber\\
& = \int_{B_\alpha} \left ( - k^2\nabla_x G_k(\vx,\vy) \times   \vEd(\vy)   -  \nabla_y \times \vEd(\vy) (\nabla_x \cdot \nabla_x  G_k(\vx,\vy)) + 
(\nabla_y \times \vEd(\vy) \cdot \nabla_x ) \nabla_x G_k(\vx,\vy)  \right ) \dif \vy \nonumber \\
& = \int_{B_\alpha} \left ( - k^2\nabla_x G_k(\vx,\vy) \times   \vEd(\vy)   +k^2   \nabla_y \times \vEd(\vy)  G_k(\vx,\vy) + 
\vD_x^2 G_k(\vx,\vy)  (\nabla_y \times \vEd(\vy) )  \right ) \dif \vy \label{eqn:intpt2}.
\end{align}
Using (\ref{eqn:intpt1}) and (\ref{eqn:intpt2}) in (\ref{eqn:intrepbd})  gives
\begin{align}
\vHd (\vx) =& -k^2 \int_{B_\alpha}   \vHd(\vy) G_k(\vx,\vy) \dif \vy
 -  \int_{B_\alpha}   \vD_x^2 G_k(\vx,\vy)  (\vHd(\vy) )\dif \vy +   \int_{B_\alpha}    \nabla_x G_k(\vx,\vy) \times \nabla_y \times  \vHd(\vy)   \dif \vy \nonumber \\
& - \im k \int_{B_\alpha} \nabla_x G_k(\vx,\vy) \times   \vEd(\vy)  \dif \vy  -\im k  \int_{B_\alpha}  \nabla_y \times \vEd(\vy)  G_k(\vx,\vy) \dif \vy \nonumber \\
&- \frac{\im}{k} 
\int_{B_\alpha} \vD_x^2 G_k(\vx,\vy)  (\nabla_y \times \vEd(\vy) )   \dif \vy .  \label{eqn:intpt3}
\end{align}
Then, using (\ref{eqn:maxwell}), we have
\begin{subequations}
\begin{align}
\nabla \times \vEd & = \im k \tilde{\mu}_r \vHd + \im k (\tilde{\mu_r}-1) \vHb,   \\
\nabla \times \vHd & = - \im k \tilde{\epsilon}_r \vEd -\im k (\tilde{\epsilon}_r-1) \vEb,  
\end{align} 
\end{subequations}
and inserting these into (\ref{eqn:intpt3}) completes the proof.
\end{proof}

\section{Proof of Theorem~\ref{thm:main}} \label{sect:proof}
We write (\ref{eqn:intrepvol}) as the sum
\begin{align}
\vHd (\vx) = &  -\im k(\epsilon_r -1) \int_{B_\alpha} \nabla_x G_k (\vx,\vy) \times \vEa (\vy) \dif \vy + (\mu_r-1) \int_{B_\alpha} \vD_y^2 (G_k(\vx,\vy)+k^2G_k(\vx,\vy)  {\mathbb I})  \vHa(\vy) \dif \vy \nonumber \\
& = \vT_1 +\vT_2 .
\end{align}
{\em Approximation of $\vT_1$}\\
Let $\vT_1 = \vT_1^a + \vT_1^b + \vT_1^c + \vT_1^d$ where
\begin{align}
\vT_1^a & := -\im k(\epsilon_r -1) \int_{B_\alpha} \nabla_x G_k (\vx,\vy) \times \left ( \vEa (\vy) - \vEb(\vy) - \alpha \vw_0 \left (  \frac{\vy -\vz}{\alpha } 
\right ) - \frac{\epsilon_r-1}{\epsilon_r }\nabla \phi_0 \right )\dif \vy  , \nonumber\\
\vT_1^b & := -\im k(\epsilon_r -1) \int_{B_\alpha} \nabla_x G_k (\vx,\vy) \times \left ( \vEb(\vy)+\nabla \phi_0 -\vF(\vy)  \right ) \dif \vy 
-\im k\frac{(\epsilon_r -1)}{\epsilon_r} \int_{B_\alpha} \nabla_x G_k (\vx,\vy) \times \nabla \phi_0   \dif \vy 
, \nonumber \\
\vT_1^c & := -\im k(\epsilon_r -1) \int_{B_\alpha} \left ( \nabla_x G_k (\vx,\vy) - \nabla_x G_k (\vx,\vz) + \vD_x^2 G_k (\vx,\vz)  (\vy-\vz) \right )\times \left (  \vF(\vy) + \alpha \vw_0 \left (  \frac{\vy -\vz}{\alpha }   \right ) \right )\dif \vy , \nonumber\\
\vT_1^d & := -\im k(\epsilon_r -1) \int_{B_\alpha} \left ( \nabla_x G_k (\vx,\vz) - \vD_x^2 G_k (\vx,\vz)  (\vy-\vz) \right )\times \left (  \vF(\vy) + \alpha \vw_0 \left (  \frac{\vy -\vz}{\alpha }   \right )  \right )\dif \vy. \nonumber 
\end{align} 
Using Theorem~\ref{thm:energyestf}, we estimate
\begin{align}
|\vT_1^a| & \le C k |\epsilon_r -1| \alpha^{3/2} \alpha^{9/2} \| \nabla \times \vEb \|_{W^{2,\infty}(B_\alpha)} \mu_r  \left (|1- \mu_r^{-1} |+ | \nu | \right ) \nonumber \\
& \le C \alpha^4 \|  \vHb \|_{W^{2,\infty}(B_\alpha)} \nonumber .
\end{align} 
Next, using (\ref{eqn:ineqebf}) and (\ref{eqn:phi}) we estimate
\begin{align}
|\vT_1^b| \le & C k |\epsilon_r -1| \alpha^{3/2}\alpha^{9/2} \| \nabla \times \vEb \|_{W^{2,\infty}(B_\alpha)}+  C k^3 \left  |1 -\frac{1}{\epsilon_r} \right | \alpha^{3/2}\alpha^{3/2}  \alpha^2 \| \vEb \|_{W^{1,\infty}(B_\alpha)} \nonumber \\
\le &  C\left (  \alpha^4 \|  \vHb \|_{W^{2,\infty}(B_\alpha)} +  k^3 \alpha^5  \left  |1 -\frac{1}{\epsilon_r} \right | \| \vEb \|_{W^{1,\infty}(B_\alpha)} \right )  \nonumber .
\end{align}
Then, using (\ref{eqn:newf}) and (\ref{eqn:neww}),
\begin{align}
|\vT_1^c| \le & C k |\epsilon_r -1| \alpha^{3/2}\alpha^{2} \alpha^{3/2} \alpha \| \nabla \times \vEb \|_{W^{2,\infty}(B_\alpha)} \nonumber \\
\le &  C \alpha^4 \|  \vHb \|_{W^{2,\infty}(B_\alpha)}\nonumber .
\end{align}
Similarly, $\vT_1^{d}= \vT_1^{d,1} + \vT_1^{d,2}+\vT_1^{d,3}+\vT_1^{d,4} +\vT_1^{d,5} + \vT_1^{d,6}$ where
\begin{align}
\vT_1^{d,1}  := & -\frac{\im k \alpha^4 (\epsilon_r -1)}{2} \sum_{i=1}^3  \nabla_x G_k (\vx,\vz)\times  \int_{B}  \left (  \ve_i \times \vxi + \vth_i \right ) \dif \vxi (\im k \vHb(\vz))_i ,  \nonumber \\
\vT_1^{d,2}  := & -  \im k \alpha^3 (\epsilon_r -1) \sum_{i=1}^3  \nabla_x G_k(\vx,\vz) \times \int_B \left (   \ve_i   + \vphi_i \right ) \dif \vxi (  \vEb(\vz))_i , \nonumber \\
\vT_1^{d,3}  := & - \frac{ \im k \alpha^5 (\epsilon_r -1)}{3} \sum_{i,j=1}^3  \nabla_x G_k(\vx,\vz) \times \int_B \left (  \xi_j   \ve_i \times \vxi  + \vpsi_{ij} \right ) \dif \vxi (  \vD_z( \im k \vHb(\vz)))_{ij} , \nonumber \\
\vT_1^{d,4} : = &\frac{\im k \alpha^5 (\epsilon_r -1)}{2} \sum_{i=1}^3  \int_B  \vD_x^2  G_k(\vx,\vz) \vxi  \times \left (  \ve_i \times \vxi + \vth_i \right ) \dif \vxi (\im k \vHb(\vz))_i , \nonumber \\
\vT_1^{d,5}  := & -  \im k \alpha^4 (\epsilon_r -1) \sum_{i=1}^3   \int_B  \vD_x^2  G_k(\vx,\vz) \vxi  \times  \left (   \ve_i   + \vphi_i \right ) \dif \vxi (  \vEb(\vz))_i , \nonumber \\
\vT_1^{d,6}  := & -  \frac{\im k \alpha^6 (\epsilon_r -1)}{3} \sum_{i,j=1}^3   \int_B  \vD_x^2  G_k(\vx,\vz) \vxi  \times \left (\xi_j   \ve_i \times \vxi  + \vpsi_{ij} \right ) \dif \vxi (  \vD_z( \im k \vHb(\vz)))_{ij} , \nonumber 
\end{align}
Defining $\vR(\vx)$ such that $ |\vR(\vx)| \le  C \left ( \alpha^4  \| \vHb \|_{W^{2,\infty}(B_\alpha)} + \alpha^4 k \left (  |\epsilon_r-1| + 
 k^2 \alpha  \left  |1 -\frac{1}{\epsilon_r} \right | \right )  \| \vEb \|_{W^{1,\infty}(B_\alpha)} \right ) ,$ we see that
 $ 
\vT_1^{a}, \vT_1^{b}, \vT_1^{c} $ and
$ \vT_1^{d,5}$ are elements of $\vR(\vx)$ and, since $ \vD_z( \im k \vHb(\vz))$ is related to $\vEb$,  $\vT_1^{d,3}$ and $\vT_1^{d,6}$ are also elements of $\vR(\vx)$.
\begin{remark}
In the case of the eddy current regime, where  (\ref{eqn:transth}) and  (\ref{eqn:transph}) simplify to
\begin{subequations} \label{eqn:transthsimp}
\begin{align}
\nabla_\xi \times \mu_r^{-1} \nabla_\xi \times \vth_i -\nu_\im  \vth_i=  &  \nu_\im \ve_i\times \vxi && \text{in $B$},\\
\nabla_\xi \times \nabla_\xi  \times\vth_i = & \vzero && \text{in $B^c$}, \\
\nabla_\xi\cdot \vth_i   = & 0 && \text{in ${\mathbb R}^3$}, \\
[\vn \times\vth_i ]_\Gamma=\vzero, \qquad [\vn \times \tilde{\mu}_r^{-1} \nabla_\xi \times\vth_i ]_\Gamma =& -2 (1-\mu_r^{-1})\vn \times \ve_i && \text{on $\Gamma$}, \\
\vth_i   = &  O(|\vxi|^{-1}) && \text{as $|\vxi| \to \infty$},
\end{align}
\end{subequations}
and 
\begin{subequations}\label{eqn:transphsimp}
\begin{align} 
\nabla_\xi \times \mu_r^{-1} \nabla_\xi \times \vphi_{i}  -\nu_\im \vphi_{i} =  &  \nu_\im \ve_i && \text{in $B$} ,\\
\nabla_\xi \times \nabla_\xi  \times\vphi_{i} = & \vzero && \text{in $B^c$}\\
\nabla_\xi\cdot \vphi_{i}   = & 0 && \text{in ${\mathbb R}^3$} , \\
[\vn \times\vphi_{i} ]_\Gamma =\vzero, \qquad [\vn \times \tilde{\mu}_r^{-1} \nabla_\xi \times\vphi_{i} ]_\Gamma  =& \vzero&& \text{on $\Gamma$}, \\
\vphi_i   = &  O(|\vxi|^{-1}) && \text{as $|\vxi| \to \infty$},
\end{align}
\end{subequations}
respectively, we find that $\vT_1^{d,1}= \vT_1^{d,2}=\vzero$ by using the above transmission problems and integration by parts. Similarly, it can be shown that $\vT_1^{d,3}=\vzero$ by simplifying the transmission problem for $\vpsi_{ij}$ in a similar way to the above.
\end{remark}

In general,  we write
\begin{align}
\vT_1 =   & -\im k\ve_j   (\nabla_x G_k(\vx,\vz))_p \varepsilon_{jpr} ( {\mathcal A}_{ri} ( \vHb(\vz))_i + {\mathcal B}_{ri}(\vEb(\vz))_i )  
\nonumber\\
&+  \ve_j  (( \vD_x^2  G_k(\vx,\vz)) _{\ell m} {\mathcal P}_{\ell m ji } (\vHb(\vz))_i + \vR(\vx) , \label{eqn:t1sum}
\end{align}
where 
\begin{subequations} \label{eqn:tenfromt1}
\begin{align}
{\mathcal A}_{ri}   := & \frac{ \im k \alpha^4 (\epsilon_r -1)}{2} \ve_r \cdot  \int_{B}  \left (  \ve_i \times \vxi + \vth_i \right ) \dif \vxi , \\
{\mathcal B}_{ri}   := & \alpha^3  (\epsilon_r -1) \ve_r\cdot  \int_B \left ( \ve_i   + \vphi_i \right ) \dif \vxi ,  \\
{\mathcal P}_{\ell m ji} := & -\frac{ k^2 \alpha^5 (\epsilon_r -1)}{2} \ve_j  \cdot \int_B \ve_{\ell} \times (\xi_m  \left (  \ve_i \times \vxi + \vth_i \right )) \dif \vxi  \nonumber \\
= & \varepsilon_{j \ell s} {\mathcal C}_{m si},  \\
{\mathcal C}_{msi} :=& \frac{1}{2} \varepsilon_{sj\ell}  {\mathcal P}_{\ell m ji}  =-\frac{k^2 \alpha^5 (\epsilon_r -1)}{2} \ve_s \cdot \int_B \xi_m  \left (  \ve_i \times \vxi + \vth_i \right ) \dif \vxi   ,
\end{align}
\end{subequations}
and the latter follows by considering the skew symmetry of ${\mathcal P}$ w.r.t  indices $j$ and $\ell$ and applying similar arguments to~\cite{LedgerLionheart2015}. Further simplifications are also possible if the eddy current approximation applies~\cite{LedgerLionheart2015} and we shall consider these later in Section~\ref{sect:altform}.

\noindent{\em Approximation of $\vT_2$} \\
Let $\vT_2 =  (\mu_r -1)( \vT_2^a + \vT_2^b + \vT_2^c+\vT_2^d)$ where
\begin{align}
\vT_2^a: = & \int_{B_\alpha} ( \vD_x^2 G_k (\vx,\vy)+k^2 G_k(\vx,\vy) {\mathbb I}) \left ( \vHa (\vy) - \mu_r^{-1} \vHb(\vy) - \mu_r^{-1} \vHb^* \left ( \frac{\vy - \vz}{\alpha}\right ) \right ) \dif \vy , \nonumber \\
\vT_2^b: = & \mu_r^{-1}  \int_{B_\alpha} ( \vD_x^2 G_k(\vx,\vy) -\vD_x^2  G_k(\vx,\vz )+k^2(G_k(\vx,\vy)-G_k(\vx,\vz)){\mathbb I}) \left (  \vHb(\vy) + \vHb^* \left ( \frac{\vy - \vz}{\alpha}\right ) \right )\dif \vy ,\nonumber \\
\vT_2^c: = & \mu_r^{-1}  \int_{B_\alpha} (\vD_x^2 G_k(\vx,\vz) +k^2G_k(v,z){\mathbb I} )\left (  \vHb(\vy) - \vHb  ( \vz ) \right ) \dif \vy , \nonumber \\
\vT_2^d: = & \mu_r^{-1}  \int_{B_\alpha} ( \vD_x^2 G_k(\vx,\vz)+k^2G_k(\vx,\vx){\mathbb I}) \left (  \vHb(\vz) + \vHb^*  \left ( \frac{\vy - \vz}{\alpha}\right ) \right ) \dif \vy \nonumber ,
\end{align}
and $\vHb^*(\vxi) = \frac{1}{\im k} \nabla \times \vw_0$. 
Following similar arguments to the above, and in~\cite{Ammari2014}, we have
\begin{align}
| \vT_2^a | \le &  C \alpha^4 \| \vHb \|_{W^{2,\infty}(B_\alpha)} , \nonumber \\
| \vT_2^b | \le &  C \alpha^4 \| \vHb \|_{W^{2,\infty}(B_\alpha)} , \nonumber \\
| \vT_2^c | \le &  C \alpha^4 \| \vHb \|_{W^{2,\infty}(B_\alpha)} , \nonumber 
\end{align}
and we can express $\vT_2^d$ as 
\begin{align}
\vT_2^d =  &  \mu_r^{-1}\alpha ^3  \left ( \sum_{i=1}^3 ( \vD_x^2 G_k(\vx,\vz) +k^2G_k(\vx,\vz){\mathbb I})\int_{B}  \left ( \ve_i + \frac{1}{2} \nabla \times \vth_i
\right ) \dif \vxi  (  \vHb(\vz))_i  \right ) \nonumber \\
& +  \mu_r^{-1}\alpha ^4 \left ( \sum_{i,j=1}^3 ( \vD_x^2 G_k(\vx,\vz) +k^2G_k(\vx,\vz){\mathbb I})\int_{B}  \left (\xi_j \ve_i + \frac{1}{3} \nabla \times \vpsi_{ij}
\right ) \dif \vxi  ( \vD_z( \vHb(\vz)))_{ij}  \right ) \nonumber \\
 =  & \mu_r^{-1}\alpha ^3  \sum_{i=1}^3  ( \vD_x^2 G_k(\vx,\vz) +k^2G_k(\vx,\vz){\mathbb I})\int_{B}  \left ( \ve_i + \frac{1}{2} \nabla \times \vth_i
\right ) \dif \vxi ( \vHb(\vz))_i  + \vR(\vx) .
 \nonumber 
\end{align}
So that
\begin{align}
\vT_2 = &  (\vD_x^2 G_k(\vx,\vz) +k^2G_k(\vx,\vz))\ve_j {\mathcal N}_{ji} (  \vHb(\vz))_i + \vR(\vx) , \label{eqn:t2sum} 
\end{align}
where
\begin{align}
{\mathcal N}_{ji}   := \alpha^3 (1-\mu_r^{-1}) \ve_j \cdot \int_B \left ( \ve_i + \frac{1}{2} \nabla \times \vth_i \right ) \dif \vxi.  \label{eq:n0}
\end{align}

\section{Simplifications and alternative forms} \label{sect:altform}

\begin{lemma} \label{lemma:btensmallk}
Using $\vphi_i = (\epsilon_r-1) \nabla \vartheta_i$ where $\vartheta_i$ solves (\ref{eqn:transphsca}) then  ${\mathcal B}_{ri} $ becomes
\begin{align}
{\mathcal B}_{ri} = \alpha^3 \left ( (\epsilon_r-1) |B| \delta_{ri} + (\epsilon_r-1)^2  \int_B \ve_r \cdot \nabla \vartheta_i \dif \vxi \right ),
\end{align}
which are also the coefficients of the symmetric P\'oyla-Szeg\"o tensor ${\mathcal T}[\alpha B, \epsilon_r]$ of an object $B_\alpha$ for a contrast $\epsilon_r$.
\end{lemma}
\begin{proof}
The proof directly follows from (\ref{eqn:tenfromt1}b) and (\ref{eqn:transphsca}).
\end{proof}

\begin{lemma} \label{lemma:sskewsym}
Up to a residual term, ${\mathcal C}_{msi}$ is skew symmetric w.r.t. $m$ and $i$ and hence
\begin{align}
{\mathcal C}_{msi} = \varepsilon_{msr} \check{{\mathcal C}}_{ri} +R_{msi}, \qquad  \check{{\mathcal C}}_{ri}: =- \frac{ \nu \alpha^3}{4} \ve_r \cdot \int_B \vxi \times ( \vth_i + \ve_i \times \vxi ) \dif \vxi,
\end{align}
where $|R_{msi}| \le C\alpha^4 k$. 
\end{lemma}
\begin{proof}
The proof follows from  applying integration by parts to (\ref{eqn:tenfromt1}) in a similar manner to the proof of Lemma 4.2 in~\cite{LedgerLionheart2015}. Compared to this proof, the additional term
\begin{equation}
R_{msi} = -\frac{\im k \alpha^4}{2} \ve_m \cdot \ \int_{\Gamma_\infty}\xi_s \vth_i \dif \xi+  \frac{\alpha^5 k^2}{2} \ve_m \cdot \int_{B \cup B^c} \xi_s \vth_i \dif \xi,
\end{equation}
arises, which is not skew symmetric w.r.t. $m$ and $i$. We estimate that $|R_{msi}| \le C \alpha^4 k$.
\end{proof}
\begin{corollary} \label{coll:sskewsym}
If $\nu$ reduces to $\nu_\im$, the transmission problem for $\vth_i(\vxi)$ provided in (\ref{eqn:transth}) reduces to (\ref{eqn:transthsimp}) and $R_{msi}=0$ since this is now identical to the case considered in Lemma 4.2 of~\cite{LedgerLionheart2015}.
\end{corollary}

By differentiation it is easily established that
\begin{align}
\nabla_x G_k( \vx, \vz)  = & \frac{(\vx-\vz)}{4\pi |\vx -\vz|^2 } \left (  \im k  - \frac{1}{  |\vx - \vz| }   \right ) e^{\im k |\vx -\vz|} , \nonumber \\
\vD_x^2 G_k( \vx, \vz)  = & \frac{1}{4\pi  } \left ( \frac{1}{  |\vx - \vz| ^3} \left (\frac{3 (\vx-\vz) \otimes (\vx-\vz) }{|\vx - \vz|^2 }- \mathbb{I} \right )    -\frac{ \im k}{  |\vx - \vz| ^2 } \left  ( \frac{3 (\vx-\vz) \otimes (\vx-\vz)}{  |\vx - \vz| ^2} - {\mathbb I}\right  ) \right .  \nonumber \\
&\left .  -\frac{k^2}{  |\vx - \vz|^3 } (\vx-\vz) \otimes (\vx-\vz)  \right ) e^{\im k |\vx -\vz|} , \nonumber
\end{align}
and, by introducing $\vr:= \vx - \vz $, $r=|\vr|$ and $\hat{\vr}= \vr /r$, these derivatives can be expressed as
\begin{align}
\nabla_x G_k( \vx, \vz)  = & \frac{\hat{\vr}}{4\pi r } \left (  \im k  - \frac{1}{  r}   \right ) e^{\im k r} , \nonumber \\
\vD_x^2 G_k( \vx, \vz)  = & \frac{1}{4\pi  } \left ( \frac{1}{r^3 } (3 \hat{\vr} \otimes \hat{\vr} - \mathbb{I} )    - \frac{\im k}{ r^2} ( 3 \hat{\vr}  \otimes \hat{\vr} - {\mathbb I} )   -\frac{k^2}{r} \hat{\vr} \otimes\hat{\vr}   \right ) e^{\im k r}.\nonumber
\end{align}
Noting that ${\mathcal C}_{msi} = \varepsilon_{msr} \check{{\mathcal C}}_{ri} +R_{msi} $ by Lemma~\ref{lemma:sskewsym} we get
\begin{align}
(\vD_x^2 G_k( \vx, \vz) )_{\ell m } \varepsilon_{j\ell s} \varepsilon_{msr} \check{{\mathcal C}}_{ri} = & - (\vD_x^2 G_k( \vx, \vz) )_{\ell m } \varepsilon_{s j\ell} \varepsilon_{smr} \check{{\mathcal C}}_{ri}\nonumber \\
= &- (\vD_x^2 G_k( \vx, \vz) )_{\ell m }  ( \delta_{jm} \delta_{\ell r} - \delta_{jr} \delta_{\ell m} ) \check{{\mathcal C}}_{ri} \nonumber\\
= &(- (\vD_x^2 G_k( \vx, \vz) )_{r j } +  (\vD_x^2 G_k( \vx, \vz) )_{m m }\delta_{jr}) \check{{\mathcal C}}_{ri} \nonumber \\
= &-( (\vD_x^2 G_k( \vx, \vz) )_{jr } +k^2 \delta_{jr}G_k(\vx,\vz)) \check{{\mathcal C}}_{ri} . \nonumber
\end{align}
Furthermore, we introduce the rank 2 tensor ${\mathcal M}$ with coefficients ${\mathcal M}_{ri} : = {\mathcal N}_{ri } - \check{{\mathcal C}}_{ri}$, which is symmetric as the following lemma shows:
\begin{lemma}
The tensor ${\mathcal M}={\mathcal N}-\check{\mathcal C}$ is complex symmetric with coefficients satisfying ${\mathcal M}_{ri} = {\mathcal M}_{ir}$.
\end{lemma}
\begin{proof}
By an application of integration parts to $\check{C}_{ri}$ we find
\begin{align}
\frac{-\check{{\mathcal C}}_{ri}}{\alpha^3} = &  \frac{1}{4 \nu}\int_B \nabla \times \mu_r^{-1} \nabla \times \vth_i \cdot \nabla \times \mu_r^{-1} \nabla \times \vth_r \dif \vxi \nonumber \\
&-\frac{1}{4} \left (\frac{\epsilon_r}{\epsilon_r-1} -1 \right ) \left ( \int_{B\cup B^c} \tilde{\mu}_r^{-1} \nabla \times \vth_i \cdot \nabla \times \vth_r \dif \vxi
-\int_{B^c} \alpha^2 k^2 \vth_i \cdot \vth_r \dif \vxi - 2[\tilde{\mu}_r^{-1}]\int_B \ve_r \cdot \nabla \times \vth_i \dif \vxi \right ) \nonumber \\
&-\frac{1}{4} \left (\frac{\epsilon_r}{\epsilon_r-1}  \right ) \left ( \int_{B\cup B^c} \tilde{\mu}_r^{-1} \nabla \times \vth_i \cdot \nabla \times \vth_r \dif \vxi
-\int_{B^c} \alpha^2 k^2 \vth_i \cdot \vth_r \dif \vxi - 2[\tilde{\mu}_r^{-1}]\int_B \ve_i \cdot \nabla \times \vth_r \dif \vxi \right ) \nonumber \\
&+\frac{1}{4} \left (\frac{\epsilon_r}{\epsilon_r-1} -1  \right ) \left ( \int_{B} \alpha^2 k^2 \epsilon_r  \vth_i \cdot  \vth_r \dif \vxi \right ), \nonumber
\end{align}
so that
\begin{align}
&\frac{{\mathcal N}_{ri} -\check{{\mathcal C}}_{r i}}{\alpha^3} =   \frac{1}{4 \nu}\int_B \nabla \times \mu_r^{-1} \nabla \times \vth_i \cdot \nabla \times \mu_r^{-1} \nabla \times \vth_r \dif \vxi + \nonumber \\
& \frac{1}{4} \int_{B\cup B^c} \tilde{\mu}_r^{-1} \nabla \times \vth_i \cdot \nabla \times \vth_r \dif \vxi + [\tilde{\mu}_r^{-1}] \int_B \delta_{r i} \dif \vxi -
\frac{\alpha^2 k^2}{4} \int_{B\cup B^c}\tilde{\epsilon}_r  \vth_i \cdot \vth_r \dif \vxi \nonumber \\
&-\frac{1}{2} \left (\frac{\epsilon_r}{\epsilon_r-1} \right ) \left ( \int_{B\cup B^c} \tilde{\mu}_r^{-1} \nabla \times \vth_i \cdot \nabla \times \vth_r \dif \vxi
-\int_{B^c} \alpha^2 k^2 \vth_i \cdot \vth_r \dif \vxi - [\tilde{\mu}_r^{-1}]\int_B (\ve_r \cdot \nabla \times \vth_i + \ve_i \cdot \nabla \times \vth_r)  \dif \vxi \right ) \nonumber \\
&+\frac{1}{4} \left (\frac{\epsilon_r}{\epsilon_r-1}   \right ) \left ( \int_{B} \alpha^2 k^2 \epsilon_r  \vth_i \cdot  \vth_r \dif \vxi \right ) \label{eqn:symnmc} ,
\end{align}
which is symmetric.
\end{proof}

It follows that an alternative form to (\ref{eqn:main}) is
\begin{align}
(\vHd(\vx)) = & -\im k \ve_j (\nabla_x G_k(\vx,\vz))_p \varepsilon_{jpr} ( {\mathcal A}_{ri} ( \vHb(\vz))_i + {\mathcal B}_{ri}(\vEb(\vz))_i )\nonumber \\
&  +\ve_j ( (\vD_x^2 G_k(\vx,\vz))_{j r}+ k^2 \delta_{jr} G_k(\vx,\vz)) {\mathcal M}_{ri} ( \vHb(\vz))_i + \vR(\vx) \nonumber  \\
 = & \frac{e^{\im k r}}{4\pi}  \left \{- \left ( - \frac{k^2}{  r} \hat{\vr} \times ( {\mathcal A}\vHb(\vz)  + {\mathcal B} \vEb(\vz)) -\frac{\im k}{ r^2 } \hat{\vr} \times  ( {\mathcal A}\vHb(\vz)  + {\mathcal B} \vEb(\vz)) \right ) \right . \nonumber \\
&+  \frac{1}{r^3} \left ( 3 \hat{\vr} \cdot ({\mathcal M} \vHb(\vz)  ) \hat{\vr} - {\mathcal M} \vHb(\vz)   \right ) - \frac{\im k}{r^2} 
\left ( 3  \hat{\vr} \cdot ({\mathcal M} \vHb(\vz)  ) \hat{\vr} - {\mathcal M} \vHb(\vz)   \right ) \nonumber \\
& \left . - \frac{k^2}{r} \left (    \hat{\vr} \cdot {\mathcal M} \vHb(\vz)   - {\mathcal M} \vHb(\vz)   \right ) \right \} +\vR(\vx) \nonumber \\
 = & \frac{e^{\im k r}}{4\pi}  \left \{
  \frac{1}{r^3} \left ( 3 \hat{\vr} \cdot ({\mathcal M} \vHb(\vz)  ) \hat{\vr} - {\mathcal M} \vHb(\vz)   \right ) \right . \nonumber \\
&  -\frac{\im k}{ r^2 } \left ( 3  \hat{\vr} \cdot ({\mathcal M} \vHb(\vz)  ) \hat{\vr} - {\mathcal M} \vHb(\vz)   -  
  \hat{\vr} \times  ( {\mathcal A}\vHb(\vz)  + {\mathcal B} \vEb(\vz)) \right )  \nonumber \\
  &  \left . - \frac{k^2}{r} \left (    \hat{\vr} \times \hat{\vr } \times ( {\mathcal M} \vHb(\vz) )
-  \hat{\vr} \times ( {\mathcal A}\vHb(\vz)  + {\mathcal B} \vEb(\vz)) \right ) \right \} +\vR(\vx) \label{eqn:mainalt},
  \end{align}
  where we have assumed that  {$(\vD_x^2 G_k( \vx, \vz) )_{\ell m } \varepsilon_{j\ell s} R_{msi} ({\vec H}_0 ({\vec z}))_i$} can be grouped with $\vR(\vx)$ and have applied $\hat{\vr} \times (\hat{\vr}  \times \va ) =\hat{\vr}  (\hat{\vr}  \cdot \va) - \va (\hat{\vr}  \cdot \hat{\vr} ) = \hat{\vr} (\hat{\vr}  \cdot \va) - \va $.   
  
\begin{remark}
Expression (\ref{eqn:mainalt}) bears close resemblance to the result in Theorem 4.1 of~\cite{LedgerLionheart2015pert}, which describes the field perturbation for wave dominated problems, where, ${\mathcal B}$ is  the P\'oyla-Szeg\"o tensor parameterised by  $\epsilon_r$, as expected, but, instead of a P\'oyla-Szeg\"o tensor parameterised by  $\mu_r$  we have the complex symmetric rank 2 tensor ${\mathcal M}$, which depends on $\mu_r$ and $\nu$. Furthermore, we have the additional term $ {\mathcal A}\vHb(\vz)$. We explore this connection further in later sections.
\end{remark}

\subsection{Quasi-static regime}
Assuming that $\alpha \ll \lambda_0$ and noting that  $(\vD_x^2 G_k( \vx, \vz) )_{jr } +k^2 \delta_{jr }G_k(\vx,\vz)  = (\vD_x^2 G_0( \vx, \vz) )_{jr  }+O(k) $ and $\nabla_x G_k( \vx, \vz) =  \nabla_x G_0( \vx, \vz) +O(k)$ as $k\to 0$ then
expression (\ref{eqn:mainalt}) reduces to
\begin{align}
(\vHd(\vx)) = & -\im k \ve_j (\nabla_x G_0(\vx,\vz))_p \varepsilon_{jpr} ( {\mathcal A}_{ri} ( \vHb(\vz))_i + {\mathcal B}_{ri}(\vEb(\vz))_i )\nonumber \\
&  +\ve_j  (\vD_x^2 G_0(\vx,\vz))_{j r}  {\mathcal M}_{ri} ( \vHb(\vz))_i + \vR(\vx), \nonumber  
 \end{align}
 and, in the near field, the dominant  response is
\begin{align}
(\vHd(\vx)) \approx &  \ve_j  (\vD_x^2 G_0(\vx,\vz))_{j r}   {\mathcal M}_{ri } ( \vHb(\vz))_i , \label{eqn:dominant}
 \end{align}
 where the coefficients of ${\mathcal M}={\mathcal N}-\check{\mathcal C}$ depend on $\epsilon_*$, $\mu_*$, $\sigma_*$, $\alpha $ and $\omega$ and are obtained from Lemma~\ref{lemma:sskewsym} and (\ref{eq:n0}) using (\ref{eqn:transth}).  {The expression (\ref{eqn:dominant}) describes the magnetic field perturbation in terms of a complex symmetric rank 2 tensor ${\mathcal M}$, extending the MPT  characterisation of objects considered in~\cite{LedgerLionheart2015,LedgerLionheart2019} for the eddy current problem, to the regime where  $\alpha \ll \lambda_0$, but without requiring $\sigma_*\gg \epsilon_* \omega$.}

\subsection{Eddy current regime}
The eddy current regime is a low frequency approximation of the Maxwell system where, in addition to the quasi-static approximation, we have $\sigma_*\gg \epsilon_* \omega$ and $\epsilon_*= \epsilon_0$ so that $\nu_\im = O(1)$ and displacement currents are neglected. Hence, (\ref{eqn:transth}) reduces to (\ref{eqn:transthsimp}) and (\ref{eqn:transph}) to (\ref{eqn:transphsimp}). Then, by application of integration by parts on ${\mathcal A}_{ri}$ and ${\mathcal B}_{ri}$ in (\ref{eqn:tenfromt1}a) and (\ref{eqn:tenfromt1}b) gives ${\mathcal A}_{ri}=0$ and ${\mathcal B}_{ri}=0$. Furthermore, by Corollary~\ref{coll:sskewsym}, ${\mathcal C}_{msi} = \varepsilon_{msr} \check{{\mathcal C}}_{ri}$ with $R_{msi}=0$ in this case and ${\mathcal M}= {\mathcal N}-\check{\mathcal C}$ becomes the MPT discussed in~\cite{LedgerLionheart2015,LedgerLionheart2019}.
Hence, 
\begin{align}
(\vHd(\vx)) = & \ve_j (\vD_x^2 G_0( \vx, \vz) )_{jr } {\mathcal M}_{ri} (\vHb(\vz))_i + \vR(\vx)  \nonumber \\
= & \frac{1}{4 \pi r^3} \left ( 3 \hat{\vr} \cdot ({\mathcal M} \vHb(\vz)  ) \hat{\vr} - {\mathcal M} \vHb(\vz)   \right ) + \vR(\vx),
\end{align}
 where the coefficients of ${\mathcal M}$ no longer depend on $\epsilon_*$ and are obtained using (\ref{eqn:transthsimp}) instead of (\ref{eqn:transth}).
 
  \begin{remark}
  In practice, all metals have $\sigma_*\gg \epsilon_0 \omega$  over all the frequencies for which $\sigma_*$ can be regarded as a constant  so that $\epsilon_r \approx \epsilon_r-1 \approx \sigma_*/(\im \epsilon_0 \omega)$ holds over these frequencies and, for sufficiently small objects, these also extend to all the frequencies for which the quasi-static approximation holds.  This means that the coefficients of ${\mathcal M}$ obtained for the eddy current model using (\ref{eqn:transthsimp}) are almost indistinguishable from those obtained using (\ref{eqn:transth}) provided that  $\sigma_*\gg \epsilon_0 \omega$ and the object is sufficiently small. As an example, we compare the diagonal coefficients of ${\mathcal M}$  obtained using the two models using the analytical solution of Wait~\cite{Wait1951} in Figure~\ref{fig:mptcomp}.  
  Of course, by making the eddy current approximation, and using (\ref{eqn:transthsimp}), then ${\mathcal A}_{ri}=0$ and, using (\ref{eqn:transphsimp}), ${\mathcal B}_{ri}=0$ indicating that (\ref{eqn:dominant}) is indeed the dominant response for the quasi-static regime for metallic objects, which holds not only in the near field. 
 \end{remark}

\begin{figure}[h]
\begin{center}
$\begin{array}{cc}
\includegraphics[width=3in]{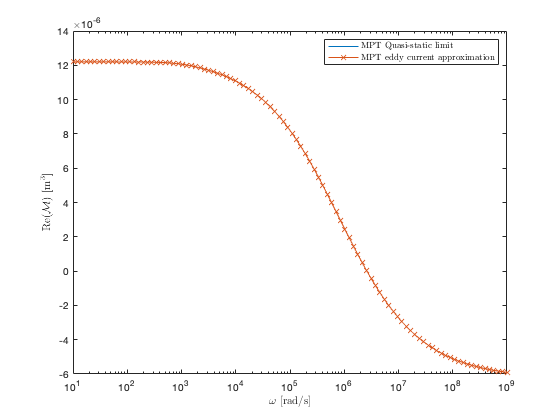} &
\includegraphics[width=3in]{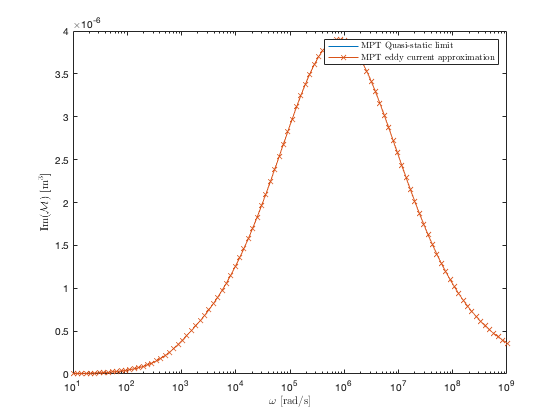} 
\end{array}$
\end{center}
\caption{Comparison of the diagonal coefficients of ${\mathcal M}$ for a conducting sphere with $\mu_r=100$, $\sigma_*=1\times 10^6$S/m and radius $\alpha=0.01$m for frequencies $1\times 10^1 \le \omega \le 1 \times 10^9$ rad/s for results obtained using the quasi-static model and eddy current approximation} \label{fig:mptcomp}
\end{figure}

\subsection{Regime with non-constant parameters}
In this regime,  {we no longer assume the parameters are constant and relax the assumption of} $\nu=O(1)$. First, we consider the case of $\sigma_*=0$ and  {require} the other parameters ($\epsilon_r=\epsilon_*/\epsilon_0$, $\mu_r$) to change so that $k \to 0$ as $\alpha \to 0$. Second, we  {require} the  parameters ($\epsilon_r=1/\epsilon_0 ( \epsilon_* + \sigma_*/(\im \omega)) $, $\mu_r$ and $k$) to change as $\alpha \to 0$.

\subsubsection{Small $k$ and $\sigma_* =0 $}
We consider small $k$ and $\sigma_*=0$, so that $\epsilon_r=\epsilon_*/\epsilon_0$. In this case, we have $\vth_i= \tilde{\vth}_i^{(0)} + O(k)$ as $k \to 0$ where $\tilde{\vth}_i^{(0)}$ solves
\begin{subequations} \label{eqn:transthwavs}
\begin{align}
\nabla_\xi \times \mu_r^{-1} \nabla_\xi \times \tilde{\vth}_i^{(0)}=  &  \vzero  && \text{in $B$},\\
\nabla_\xi \times \nabla_\xi \times  \tilde{\vth}_i^{(0)} = & \vzero && \text{in $B^c$},\\
\nabla_\xi\cdot \tilde{\vth}_i^{(0)}   = & 0 && \text{in ${\mathbb R}^3$},\\
[\vn \times\tilde{\vth}_i^{(0)} ]_\Gamma=\vzero, \qquad [\vn \times \tilde{\mu}_r^{-1} \nabla_\xi \times\tilde{\vth}_i^{(0)} ]_\Gamma =& -2 (1-\mu_r^{-1})\vn \times \ve_i && \text{on $\Gamma$},\\
\tilde{\vth}_i^{(0)}  = & O(|\vxi|^{-1}|) && \text{as $|\vxi| \to \infty$},
\end{align}
\end{subequations}
which allow us to deduce the following about the coefficients of ${\mathcal M}$:
\begin{lemma} \label{lemma:smallkmten}
For small $k$ and $\sigma_*= 0$, $\vth_i = \tilde{\vth}_i^{(0)} + O(k)$  as $k \to 0$ and ${\mathcal M}_{ri} = {\mathcal N}_{ri} -\check{{\mathcal C}}_{ri}$ reduces to ${\mathcal M}_{ri} = {\mathcal N}^{(0)}_{ri} +O(k^2) $ where
\begin{align}
{\mathcal N}^{(0)}_{r i} = \frac{\alpha^3 }{4} \int_{B\cup B^c} \tilde{\mu}_r^{-1} \nabla \times \tilde{\vth}_i^{(0)} \cdot \nabla \times \tilde{\vth}_r^{(0)} \dif \vxi + [\tilde{\mu}_r^{-1}] \int_B \delta_{r i} \dif \vxi ,
\end{align}
which are also the coefficients of the P\'oyla-Szeg\"o tensor ${\mathcal T}[\alpha B,\mu_r]$ of an object $B_\alpha$ for a contrast $\mu_r$ as defined e.g. in Lemma 3 of~\cite{LedgerLionheart2016}. 
\end{lemma}
\begin{proof}
For the transmission problem  (\ref{eqn:transthwavs})  we can show that
\begin{align}
 \int_{B\cup B^c} \tilde{\mu}_r^{-1} \nabla \times \tilde{\vth}_i^{(0)}  \cdot \nabla \times \tilde{\vth}_r^{(0)} \dif \vxi
=  [\tilde{\mu}_r^{-1}]\int_B (\ve_r \cdot \nabla \times \tilde{\vth}_i^{(0)} + \ve_i \cdot \nabla \times\tilde{\vth}_r^{(0)})  \dif \vxi , \label{eqn:intptstheta0}
\end{align}
which is obtained  by integrating by parts $\int_{B\cup B^c} \left ( \nabla \times {\tilde{\mu}}_r^{-1} \nabla \times \tilde{\vth}_i^{(0)}  \cdot \tilde{\vth}_r^{(0)}  + \nabla \times \tilde{\mu}_r^{-1} \nabla \times\tilde{\vth}_r^{(0)} \cdot \tilde{\vth}_i^{(0)} \right )  \dif \vxi =0$. Then, for small $k$, we have $\vth_i= \tilde{\vth}_i^{(0)} + O(k)$ and, by using (\ref{eqn:intptstheta0}) in (\ref{eqn:symnmc}), we
find that $  {\mathcal N}_{ri} -\check{{\mathcal C}}_{ri}$ reduces to $ {\mathcal N}^{(0)}_{ri}$ and this is also equivalent to the coefficients of the P\'oyla-Szeg\"o tensor ${\mathcal T}[\alpha B, \mu_r]$ of an object $B$ for a contrast ${\mu}_r$, as shown by Theorem 3.2 of ~\cite{LedgerLionheart2019} and Lemma 3 of~\cite{LedgerLionheart2016}. Note that Lemma 3 of ~\cite{LedgerLionheart2016} considered the question of connectedness of $B$ and showed this reduction holds independent of the first Betti number of $B$.
\end{proof}

By integration by parts and then using $\vth_i= \tilde{\vth}_i^{(0)} + O(k)$ we get
\begin{align}
{\mathcal A}_{ri}   := & \frac{ \im  \alpha^3  }{2\alpha k} \ve_r \cdot  \int_{B}  \left (  \nabla_\xi\times\mu_r^{-1} \nabla_\xi \times \vth_i - k^2 \alpha^2 \vth_i  \right ) \dif \vxi \nonumber \\
=&-\frac{ \im   \alpha^4 k }{2} \ve_r \cdot  \int_{B\cup B^c}   \vth_i   \dif \vxi = O(k),
\end{align}
as $k \to 0$.

Using the above results and Lemma~\ref{lemma:btensmallk} means that (\ref{eqn:mainalt}) now becomes
\begin{align}
(\vHd(\vx)) = & \frac{e^{\im k r}}{4\pi}  \left \{
  \frac{1}{r^3} \left ( 3 \hat{\vr} \cdot ({\mathcal T} [\alpha B, \mu_r] \vHb(\vz)  ) \hat{\vr} - {\mathcal T} [\alpha B, \mu_r] \vHb(\vz)   \right ) \right . \nonumber \\
&  -\frac{\im k}{ r^2 } \left ( 3  \hat{\vr} \cdot ({\mathcal T} [\alpha B, \mu_r] \vHb(\vz)  ) \hat{\vr} - {\mathcal T} [\alpha B, \mu_r]  \vHb(\vz)   -  
  \hat{\vr} \times  (  {\mathcal T} [\alpha B, \epsilon_r] \vEb(\vz)) \right )  \nonumber \\
  &  \left . - \frac{k^2}{r} \left (    \hat{\vr} \times \hat{\vr } \times ( {\mathcal T} [\alpha B, \mu_r] \vHb(\vz) )
-  \hat{\vr} \times (  {\mathcal T} [\alpha B, \epsilon_r] \vEb(\vz)) \right ) \right \} +\vR(\vx) \label{eqn:mainalt2},
  \end{align}
which agrees with the known results for low-frequency scattering from dielectric and permeable bodies (e.g.~\cite{LedgerLionheart2015pert}).

\subsubsection{Small $\alpha$ and small $\nu$}
We now consider the case of small $\alpha$ and  {require} $k,\epsilon_r,\mu_r$  to change  {with $\alpha$} such that  $\nu$ is also small. For small $\alpha, \nu$ then $\vth_i= \tilde{\vth}_i^{(0)} + O(\alpha)$  as $\alpha \to 0$ where $\tilde{\vth}_i^{(0)}$ solves (\ref{eqn:transthwavs}),
which allow us to show the following about the coefficients of ${\mathcal M}$:
\begin{lemma}
For small $\alpha$ and $\nu $, $\vth_i= \vth_i^{(0)} +O(\alpha)$ as $\alpha \to 0$ and ${\mathcal M} _{ri} = {\mathcal N}_{ri} -\check{\mathcal C}_{ri }$ reduces to ${\mathcal M}_{ri} = {\mathcal N}^{(0)}_{ri} +O(\alpha^5) $ where
\begin{align}
{\mathcal N} ^{(0)}_{ri} = \frac{\alpha^3 }{4} \int_{B\cup B^c} \tilde{\mu}_r^{-1} \nabla \times \tilde{\vth}_i^{(0)} \cdot \nabla \times \tilde{\vth}_r^{(0)} \dif \vxi + [\tilde{\mu}_r^{-1}] \int_B \delta_{ri} \dif \vxi, 
\end{align}
which are also the coefficients of the P\'oyla-Szeg\"o tensor ${\mathcal T} [\alpha B, \mu_r]$ of an object $B_\alpha$ for a contrast $\mu_r$.
\end{lemma}
\begin{proof}
The proof is similar to the proof of  Lemma~(\ref{lemma:smallkmten}) and  again uses (\ref{eqn:intptstheta0}). For small $\alpha, \nu$, we have $\vth_i= \tilde{\vth}_i^{(0)} + O(\alpha)$ as $\alpha \to 0$ and find that $  {\mathcal N}_{ri} -\check{{\mathcal C}}_{ri}$ reduces to $ {\mathcal N}^{(0)}_{ri}$.
\end{proof}

Similarly, by integration by parts and then using $\vth_i= \tilde{\vth}_i^{(0)} + O(\alpha)$, we get
\begin{align}
{\mathcal A}_{ri}   := & \frac{ \im  \alpha^3  }{2\alpha k} \ve_r \cdot  \int_{B}  \left (  \nabla_\xi\times\mu_r^{-1} \nabla_\xi \times \vth_i - k^2 \alpha^2 \vth_i  \right ) \dif \vxi \nonumber \\
=&-\frac{ \im   \alpha^4 k }{2} \ve_r \cdot  \int_{B\cup B^c}   \vth_i   \dif \vxi = O(\alpha^4),
\end{align}
as $\alpha \to 0$.

Using the above results and Lemma~\ref{lemma:btensmallk}  means that (\ref{eqn:mainalt}) now becomes
\begin{align}
(\vHd(\vx)) = & \frac{e^{\im k r}}{4\pi}  \left \{
  \frac{1}{r^3} \left ( 3 \hat{\vr} \cdot ({\mathcal T}[\alpha B, \mu_r] \vHb(\vz)  ) \hat{\vr} - {\mathcal T} [\alpha B,\mu_r] \vHb(\vz)   \right ) \right . \nonumber \\
&  -\frac{\im k}{ r^2 } \left ( 3  \hat{\vr} \cdot ({\mathcal T} [\alpha B,\mu_r] \vHb(\vz)  ) \hat{\vr} - {\mathcal T} [\alpha B,\mu_r]  \vHb(\vz)   -  
  \hat{\vr} \times  (  {\mathcal T} [\alpha B,\epsilon_r] \vEb(\vz)) \right )  \nonumber \\
  &  \left . - \frac{k^2}{r} \left (    \hat{\vr} \times \hat{\vr } \times ( {\mathcal T} [\alpha B,\mu_r] \vHb(\vz) )
-  \hat{\vr} \times (  {\mathcal T} [\alpha B,\epsilon_r] \vEb(\vz)) \right ) \right \} +\vR(\vx) , \label{eqn:mainaltsmallalpha}
  \end{align}
which agrees with the known results for  scattering from small bodies~\cite{ammari2001,ammari2005}. 

\begin{remark}
While the forms for $\vHd(\vx)$ derived in (\ref{eqn:mainalt}) and (\ref{eqn:mainaltsmallalpha}) look very similar, the former holds for small $k$ and requires $\sigma_*=0$ and the latter is for the case of small $\alpha$ and small $ \nu$, but does not require $\sigma_*=0$.
\end{remark}

\section*{Acknowledgements}

Paul D. Ledger gratefully acknowledges the financial support received from EPSRC in the form of grant  EP/V009028/1. William R. B. Lionheart gratefully acknowledges the financial support received from EPSRC in the form of grant  EP/V009109/1 and would like to thank the Royal Society for the financial support received from a Royal Society Wolfson Research Merit Award and  a Royal Society  Global Challenges Research Fund grant CH160063.

\section*{Declaration}

The authors have no competing interests to declare that are relevant to the content of this article.

\bibliographystyle{plainurl}
\bibliography{paperbib}
\end{document}